\documentclass[12pt,oneside]{amsart}
\usepackage[utf8]{inputenc}
\usepackage[T1]{fontenc}
\usepackage[BCOR=0mm]{typearea}
\usepackage{amsmath,amssymb,amsaddr}
\usepackage{graphicx}
\usepackage[only,llbracket,rrbracket]{stmaryrd}
\usepackage{enumerate}
\usepackage[centercolon]{mathtools}
\usepackage[usenames,dvipsnames,svgnames]{xcolor}
\usepackage{hyperref}
\usepackage{nicefrac}
\usepackage{microtype}
\usepackage{lmodern}

\def\Xint#1{\mathchoice
{\XXint\displaystyle\textstyle{#1}}%
{\XXint\textstyle\scriptstyle{#1}}%
{\XXint\scriptstyle\scriptscriptstyle{#1}}%
{\XXint\scriptscriptstyle\scriptscriptstyle{#1}}%
\!\int}
\def\XXint#1#2#3{{\setbox0=\hbox{$#1{#2#3}{\int}$ }
\vcenter{\hbox{$#2#3$ }}\kern-.59\wd0}}
\def\avint{\Xint-}

\newcommand{\edge}{{\scriptstyle\mid}}
\newcommand{\grad}{\nabla}
\renewcommand{\div}{\grad\cdot}

\newcommand{\N}{\mathbf{N}}

\newcommand{\R}{\mathbf{R}}

\DeclareMathOperator{\diam}{diam}

\def\opt{{\mathrm{opt}}}
\newcommand{\zo}{\zeta_{\opt}}
\newcommand{\po}{\pi_{\opt}}
\newcommand{\dt}{{\delta}}
\newcommand{\rhoh}{\rho_{\dt,h}}

\DeclareMathOperator{\Hess}{Hess}

\newcommand{\Ha}{\ensuremath{\mathcal{H}}}

\newcommand{\io}{\int_{\Omega}}

\newcommand{\llb}{\llbracket}
\newcommand{\rrb}{\rrbracket}

\newcommand{\eps}{\varepsilon}

\newtheorem{proposition}{Proposition}
\newtheorem{theorem}{Theorem}
\newtheorem{lemma}{Lemma}

\DeclarePairedDelimiter{\abs}{\lvert}{\rvert}
\DeclarePairedDelimiter{\norm}{\lVert}{\rVert}
\DeclarePairedDelimiter{\bra}{(}{)}

\DeclarePairedDelimiter{\set}{\{}{\}}

\newcommand{\cT}{\mathcal{T}}

\DeclareMathOperator{\I}{I}
\DeclareMathOperator{\II}{II}
\DeclareMathOperator{\III}{III}
\DeclareMathOperator{\IV}{IV}

\definecolor{darkblue}{rgb}{0,0,0.6}
\hypersetup{
    pdftitle={Analysis of the implicit upwind finite volume scheme with rough coefficients},
    pdfauthor={André Schlichting and Christian Seis},
    colorlinks=true,    
    linkcolor=darkblue, 
    citecolor=darkblue, 
    filecolor=darkblue, 
    urlcolor=darkblue   
}

\begin{document}

\title[Analysis of the implicit upwind scheme with rough coefficients]{Analysis of the implicit upwind finite volume scheme with rough coefficients}

\author{Andr\'e Schlichting \and Christian Seis}

\address{Institut f\"ur Angewandte Mathematik, Universit\"at Bonn}
\email{schlichting@iam.uni-bonn.de}
\email{seis@iam.uni-bonn.de}

\date{\today}

\begin{abstract}
We study the implicit upwind finite volume scheme for numerically approximating the linear continuity equation in the low regularity DiPerna--Lions setting. That is, we are concerned with advecting velocity fields that are spatially Sobolev regular and data that are merely integrable. We prove that on unstructured regular meshes the rate of convergence of approximate solutions generated by the upwind scheme towards the unique distributional solution of the continuous model is at least~$\nicefrac{1}{2}$. The numerical error is estimated in terms of logarithmic Kantorovich--Rubinstein distances and provides thus a bound on the rate of weak convergence.
\end{abstract}

\maketitle

\section{Introduction}

The continuity equation is one of the most fundamental and at the same time most elementary partial differential equation with applications in a wide range of problems from physics, engineering, biology or social science. It describes the conservative transport of a quantity (e.g., a mass or number density, temperature, concentration, or tracer) by a given velocity field. In many important examples, the velocity field itself is related to the actual configuration of this quantity by means of a momentum-type equation, thermodynamic law or other basic principles. In this paper, we are interested in the purely linear model in which such a feedback of the actual configuration on the velocity field is neglected. For the sake of a larger applicability, we will allow for external sources and sinks.

We are thus concerned with  the linear inhomogeneous continuity equation in a bounded Lipschitz domain $\Omega $ in $\R^d$ and time interval $(0,T)$, that is,
\begin{equation}
\label{1}
\left\{
\begin{aligned}
  \partial_t \rho + \div\bra*{ u\rho} &= f &&\text{in } (0,T)\times \Omega,\\
  \rho(0,\cdot) &= \rho^0 &&\text{in }\Omega.
\end{aligned}
\right.
\end{equation}
Here, $\rho$ is the evolving (scalar) quantity with initial configuration $\rho^0$, $u$ the velocity field and  $f$ is the source-sink distribution. We will often refer to $\rho$ as a density even if it may take negative values.

We will assume in this paper that there is no loss of mass by transport across the boundary, that is, we suppose that the velocity field is tangential at $\partial\Omega$,
\begin{equation}
\label{2}
u\cdot \nu = 0\qquad\text{on } \partial \Omega,
\end{equation}
if $\nu$ denotes the outer unit normal vector along the boundary. 
Under this hypothesis we can apply a standard transformation which ensures that the total sources and sinks are balanced. We will accordingly demand that
\begin{equation}
\label{11}
\int_{\Omega} f(t,x)\, dx = 0\qquad\text{for all }t\in(0,T).
\end{equation}
The previous two assumptions together imply that the continuity equation is indeed conservative in the sense that
\begin{equation*}
\io \rho(t,x)\, dx = \io\rho^0(x)\, dx\qquad\text{for all }t\in(0,T).
\end{equation*}
It is obvious that this identity is formally true and it indeed holds under the assumptions of this paper which will be specified in the following.

In many relevant applications, for instance, in turbulent transport of mass or heat, neither the velocity field nor the transported density are expected to be regular functions. The minimal mathematical requirement for guaranteeing well-posedness of the Cauchy problem \eqref{1} is that $u$ has Sobolev or bounded variation ($BV$) regularity in the spatial variable. This is the setting studied in the ground-breaking papers of DiPerna and Lions~\cite{DiPernaLions89} and Ambrosio~\cite{Ambrosio04}. We will focus on the case of Sobolev vector fields, which is the setting originally studied by DiPerna and Lions, and will thus assume that
\begin{equation}
\label{20}
u\in L^1((0,T);W^{1,p}(\Omega))\qquad\text{and}\qquad (\div u)^-\in L^1((0,T);L^{\infty}(\Omega))
\end{equation}
for some $p\in(1,\infty]$. Here the  superscript minus sign indicates the negative part of the divergence. While the first hypothesis implies the so-called renormalization property, which, in a certain sense, is the verification of the chain rule for solutions of \eqref{1}, the second hypothesis yields the validity of the a priori estimate
\begin{equation}
\label{18}
\|\rho\|_{L^{\infty}(L^q)} \le \Lambda^{1-\frac{1}{q}} \bra*{\|\rho^0\|_{L^q} + \|f\|_{L^1(L^q)}},
\end{equation}
provided that the right-hand side is finite, where $\Lambda := \exp\bra*{\|(\div u)^-\|_{L^1(L^{\infty})}}$ is the compressibility constant associated with $u$.
Here and in the following, we use the shorter notation $L^r(X)$ instead of $L^r((0,T);X)$ for any Banach space $X$ and $r\in[1,\infty]$.
In principle, the results in \cite{DiPernaLions89,Ambrosio04} allow for the consideration of densities that are merely summable.
In this case, the relevant notion of solutions is that of renormalized solutions. In the present paper, we are interested in the smaller class of distributional solutions, which are well-defined if
\begin{equation}
\label{21}
\rho\in L^{\infty}((0,T);L^q(\Omega))\qquad\text{where}\qquad\nicefrac1p+\nicefrac1q=1.
\end{equation}
This integrability assumption is consistent with the a priori estimate  \eqref{18}.

The aim of this paper is to provide an error estimate for the implicit  upwind finite volume approximation on unstructured meshes of distributional solutions to the continuity equation \eqref{1} in the DiPerna--Lions setting  \cite{DiPernaLions89}. The upwind scheme is the most classical stable monotone and mass preserving numerical method for hyperbolic conservation laws, whose underlying idea is the numerical approximation of the upstream transport of cell averages. Convergence rates for upwind methods were intensively studied in the past forty years, starting with the pioneering work by Kuznetsov in 1976 \cite{Kuznetsov76}. In the linear case with Lipschitz regular vector fields and $BV$ or $H^1$ initial data, the optimal results on unstructured meshes are due to Johnson and Pitk\"aranta \cite{JohnsonPitkaranta86}\footnote{In fact, Johnson and Pitk\"aranta consider the more general discrete Galerkin approximation.}, Vila and Villedieu~\cite{VilaVilledieu03}, Merlet and Vovelle \cite{MerletVovelle07,Merlet07}, Delarue and Lagouti\`ere \cite{DelarueLagoutiere11}, and Aguillon and Boyer \cite{AB16}. All these works provide $\nicefrac{1}{2}$~convergence rates in terms of the mesh size~$h$.
Concerning the optimality of these results, we refer to the papers by Peterson \cite{Peterson91} and Tang and Teng~\cite{TangTeng95}. The case of lower regularity data was treated only very recently. Under a one-sided Lipschitz condition on the vector field, Delarue, Lagouti\`ere and Vauchelet estimated the approximation error by~$\mathcal{O}(h^{\nicefrac{1}{2}})$ for arbitrary measure-valued solutions \cite{DelarueLagoutiereVauchelet16}. Using similar (probabilistic) techniques, the authors of the present paper derived the same rate for distributional solutions in the class \eqref{21} under the assumption~\eqref{20}, cf.\ \cite{SchlichtingSeis16a}. Unfortunately, both works~\cite{DelarueLagoutiereVauchelet16} and~\cite{SchlichtingSeis16a} are restricted to Cartesian meshes.

In the present work, convergence rates for upwind methods on unstructured meshes in the low regularity framework are obtained for the first time. As in the classical setting, we prove a $\mathcal{O}(h^{\nicefrac{1}{2}})$ error bound in the case of rough densities \eqref{21} and Sobolev vector fields~\eqref{20}. Less importantly, we treat the implicit scheme instead of the explicit scheme studied earlier in \cite{SchlichtingSeis16a}, but we're convinced that the explicit one could be handled in a similar fashion. The step from Cartesian to arbitrary meshes appears to be substantial: While on Cartesian meshes the upwind scheme allows for an interpretation as a finite difference method, the scheme is genuinely of finite volume type on unstructured meshes. In our previous work \cite{SchlichtingSeis16a} (and that of Delarue, Lagouti\`ere and Vauchelet  \cite{DelarueLagoutiereVauchelet16}) the Cartesian mesh geometry allowed for a comparison of the Lagrangian flow associated with the continuous problem with a stochastic flow associated with a probabilistic (Markov chain) interpretation of the upwind scheme.
On unstructured meshes this comparison seems to fail. Moreover, an adaptation of the construction used for the transport equation with Lipschitz regular vector field by Delarue and Lagouti\`ere~\cite{DelarueLagoutiere11} to the continuity equation in the low regularity regime is not apparent.
To overcome this difficulty, in the present work we choose to work mostly with the Eulerian specification of the problem rather than the Lagrangian one considered in \cite{DelarueLagoutiereVauchelet16,SchlichtingSeis16a}. Instead of comparing the flows, we directly estimate the distance of  continuous and approximate solution. For this purpose, we derive a number of new optimal stability estimates for  continuity equations that built up on and further extend techniques recently  established in \cite{Seis16a,Seis16b}. The Lagrangian formulation of transport  enters our analysis only through the superposition principle.

The focus on  density functions of low regularity comes along with a change in topology: The  results established in \cite{DelarueLagoutiereVauchelet16,SchlichtingSeis16a} and the present paper quantify the rate of \emph{weak} convergence (of measures). This is contrasted by the ``classical'' setting with $BV$ or $H^1$ densities considered in~\cite{Kuznetsov76,JohnsonPitkaranta86,VilaVilledieu03,MerletVovelle07,Merlet07,DelarueLagoutiere11}, where optimal rates in strong Lebesgue norms can be proved. In~\cite{SchlichtingSeis16a} we show that this change of topology is not at all a pathology of the applied method. In fact, quite elementary examples indicate that, firstly, no rates exist that are uniform in the initial data and, secondly, rate~$\nicefrac{1}{2}$ weak convergence is (almost) sharp in the sense that for any small~$\eps$ there exists configurations that generate rates of at most~$\nicefrac{1}{2}+\eps$. In this respect, the qualitative convergence results in strong norms for upwind schemes with rough coefficients obtained by Walkington \cite{Walkington05} and Boyer \cite{Boyer12} are optimal, too.

\emph{The paper is organized as follows}: Section \ref{S:scheme} contains the precise definition of the implicit upwind finite volume scheme. In Section \ref{S:results} we present and discuss our main results. Section \ref{S:KR} provides an overview on relevant facts about Kantorovich--Rubinstein distances. In Section \ref{S:properties} we derive stability and weak $BV$ estimates for the upwind scheme. The final Section \ref{S:estimates} is devoted to the error analysis. We conclude this paper with a short appendix on generalized means.

\section{The upwind scheme: I. Definition}\label{S:scheme}
The upwind scheme is a classical finite volume method for approximating hyperbolic conservation laws.
It approximates, roughly speaking, the evolution of volume averages by means of the flux over their boundaries, where it only uses the upstream values to update neighboring cell averages in each time step.
See also the monograph~\cite{EymardGallouetHerbin00} for references and further details.

In this section, we describe the upwind scheme for the continuity equation with arbitrary source-sink distribution \eqref{1}. Since there is no mass flux across the domain boundary, cf.\ \eqref{2}, there is no need to restrict the geometry of  $\Omega$ in what follows.

We consider as tessellation $\cT$ of $\Omega$ a finite disjoint polyhedral covering of $\Omega$. That is a set of finite many closed connected sets $K$ with disjoint interiors and such that $\overline \Omega = \bigcup_{K\in\cT}K$. Moreover, for each $K\in \cT$ exists a polyhedral set $K'$ such that $K=K'\cap \overline\Omega$. These sets~$K$ are called control volumes or simply cells. The interior boundary of each cell, that is $\partial K\setminus \partial \Omega$, is the union of finitely many flat closed and connected $d-1$ dimensional faces. We write $K\sim L$ whenever  $K$ and $L$ are two neighboring cells  and  we denote by~$K\edge L$ the joint edge $K\cap L$. The unit normal vector on the edge $K\edge L$ pointing from~$K$ to~$L$ will be denoted by $\nu_{KL}$, so that $\nu_{KL} = -\nu_{LK}$. The definition of the scheme involves the edge size $|K\edge L|$ and the volume $|K|$, where, by a common abuse of notation, we have used $|\cdot|$ for both, the $d-1$ dimensional Hausdorff measure and the $d$ dimensional Lebesgue measure. Finally, the mesh size of the tessellation $\cT$ is defined as the maximal cell diameter,
\[
h := \max_{K\in\cT} \diam K.
\]
For our convergence analysis, we have to assume a certain regularity of the mesh which ensures that standard geometric constants can be chosen independently of the tessellation, and in particular, independently of the mesh size $h$. To be more specific, we require that the constants $C$ in the trace and Poincaré estimates
\begin{equation}
\label{3}
\begin{aligned}
\|\psi\|_{L^1(\partial K)} &\leq C\bra[\big]{\|\grad \psi\|_{L^1(K)} + h^{-1} \|\psi\|_{L^1(K)}}, \\
 \norm{ \psi - \psi_K}_{L^1(K)} &\leq C \, h \, \norm{\nabla \psi }_{L^1(K)},
\end{aligned}
\end{equation}
are uniform in $K\in \cT$ and $h>0$.
Hereby, $\psi_K:= \avint_K \psi \, dx$ is the average of $\psi$ on $K$.
For a proof of the above trace and Poincaré estimates with implicit constants, we refer to~\cite[Chapter 4.3 and 4.5]{EvansGariepy92}. It is worth to note that for the particular choice~$\psi\equiv 1$, the trace estimate implies the isoperimetric property
\begin{equation}\label{e:iso:mesh}
\frac{|\partial K|}{|K|} \le \frac{C}h,
\end{equation}
which in turn guarantees that the volume of each cell $K\in \cT$ is of order $h^d$, and its surface area is of order  $h^{d-1}$, uniformly in $h$.

As we are concerned with an implicit scheme, the time step size $\dt$ can be chosen independently of the mesh size. For convenience, we assume that $\dt$ is fixed in each step, and thus, we can write $t^n = n\, \dt$ for the $n$-th time step. The final time step is the largest integer $N$ satisfying $N\, \dt \le T$. We will sometimes write $\llb 0,N\rrb := \{0,1,\dots,N\}$.

The upwind scheme approximates solutions by cell averages. On each cell $K\in \cT$, the initial datum $\rho^0$ is thus approximated by its average
\begin{equation}\label{26}
\rho_K^0 := \avint_K\rho^0\, dx.
\end{equation}
Similarly, in each time interval $[t^n,t^{n+1})$ and cell $K$, the source term is replaced by
\begin{equation}
\label{27}
f^n_K := \avint_{t^{n}}^{t^{n+1}} \avint_K f\, dx\,dt.
\end{equation}
The scheme takes into account the net fluxes across the cell faces. The average normal velocity from one cell $K$ to a neighboring cell $L\sim K$  is defined by
\begin{equation}
\label{28}
u_{KL}^n := \avint_{t^n}^{t^{n+1}} \avint_{K\edge L} u\cdot \nu_{KL} \, d\Ha^{d-1} \, dt.
\end{equation}
Here, $\Ha^{d-1}$ denotes the $d-1$ dimensional Hausdorff measure. Notice that these quantities are well-defined thanks to the trace theorem for Sobolev functions (see, for instance, \cite[Chapter 4.3]{EvansGariepy92}) and assumption \eqref{20}. By the sign convention of cell normals it holds $u_{KL}^n = -u_{LK}^n$. We need to distinguish between the fluxes inwards and outwards each reference cell. We thus write $u_{KL}^{n\pm}  = \bra*{u_{KL}^n}^{\pm}$ with $(q)^+:=\max\set{0,q}$ and $(q)^-:=\max\set{0,-q}$ denoting the positive and the negative part of a quantity $q\in\R$, respectively.

The implicit upwind finite volume scheme for the linear continuity equation \eqref{1} now reads
\begin{equation}\label{scheme}
\frac{\rho_K^{n+1} - \rho_K^n}{\dt}
+\sum_{L\sim K} \frac{\abs{K\edge L}}{\abs{K}}\; \bra[\big]{u_{KL}^{n+}\,\rho_K^{n+1} - u_{KL}^{n-}\,\rho_L^{n+1}}=f_K^n
\end{equation}
for every $n\in \llb 0,N\rrb$ and $K\in \cT$. Each $\rho_K^n$ can be thus thought of as the approximate volume average of the exact solution at time $t^n$.
The \emph{approximate solution} $\rho_{\dt,h}$ is accordingly given by
\begin{equation}
\label{29}
\rho_{\dt,h} (t,x) := \rho^n_K\qquad\text{for a.e.\ }(t,x)\in \bigl[t^n,t^{n+1}\bigr)\times K.
\end{equation}
In the case $n=0$, we also write $\rho_h^0 = \rho_{\dt,h}^0 = \rho_{\dt, h}(0,\cdot)$. The approximate source term~$f_{\dt, h}$ is defined analogously.

For later purposes it is beneficial to remark that the upwind scheme \eqref{scheme} can  equivalently be formulated as
  \begin{equation}\label{scheme2}
    \frac{\rho^{n+1}_K - \rho^{n}_K}{\dt} + \sum_{L\sim K} \frac{\abs{K\edge L}}{\abs{K}} \, u_{KL}^n \,\frac{\rho_K^{n+1} + \rho_L^{n+1}}{2} + \sum_{L\sim K} \frac{\abs{K\edge L}}{\abs{K}} \, \abs*{u_{KL}^n} \, \frac{\rho_K^{n+1} - \rho_L^{n+1}}{2} = f_K^n ,
  \end{equation}
which follows from the identities  $u_{KL}^{n+} = \frac{1}{2}\bra[\big]{ \abs{u_{KL}^n} + u_{KL}^n}$ and $u_{KL}^{n-} = \frac{1}{2}\bra[\big]{ \abs{u_{KL}^n} - u_{KL}^n}$. Well-posedness (cf.~Lemma~\ref{lem:wellposed}) and quantified stability (cf.~Lemma~\ref{lem:stability-energy}) analogous to~\eqref{18} follow under the additional condition that $\dt\leq \dt_{\max}$, where for some $\kappa>1$ the maximal time step size $\dt_{\max} = \dt_{\max}(\kappa)$ is such that
\begin{equation}
\label{43}
\frac{q-1}{q} \int_I \|(\div u)^-\|_{L^{\infty}}\, dt \le \frac{\kappa-1}{\kappa}\qquad \text{for all intervals $I$ of length } \dt_{\max}(\kappa).
\end{equation}
A similar condition on the time step size was introduced earlier by Boyer, see \cite[Eq.~(3.1)]{Boyer12}. We introduce the constant $\kappa$ in order to quantify how close the approximate solutions~$\rho_{\delta,h}$ get to satisfying the a priori estimate \eqref{18}. In fact, we are able to prove a substitute for~\eqref{18} satisfied by $\rho_{\delta,h}$ in Lemma~\ref{lem:stability-energy} below, in which under condition~\eqref{43} on the maximal time step size the exponent $1-\nicefrac{1}{q}$ on the compressibility constant in~\eqref{18} is replaced by $\kappa\bra{1-\nicefrac{1}{q}}$.
Notice that in the case of divergence-free vector fields, we can set $\kappa=1$.

\section{Main results}\label{S:results}

Our main result is an estimate on the numerical error generated by the implicit upwind finite volume scheme \eqref{scheme} for the continuity equation \eqref{1}. Before stating the result, we recall or specify the underlying  hypotheses. We suppose that the initial configuration~$\rho^0$ and the source-sink distribution~$f$ are integrable functions such that
\begin{equation}
\label{24}
\rho^0\in L^q(\Omega),\qquad\text{and}\qquad  f\in L^1((0,T);L^q(\Omega))\cap L^{\infty}((0,T);W^{-1,1}(\Omega)),
\end{equation}
for some $q\in (1,\infty)$. Here $W^{-1,1}(\Omega)$ is the space that is dual to the homogeneous Lipschitz space $\dot W^{1,\infty}(\Omega)$. For the advecting velocity field $u$ we suppose that
\begin{equation}
\label{25}
u\in L^1((0,T);W^{1,p}(\Omega))\qquad\text{with}\qquad (\div u)^- \in L^1((0,T);L^{\infty}(\Omega)) ,
\end{equation}
where $p\in(1,\infty)$ is such that $\nicefrac1p+\nicefrac1q=1$. Slightly modifying  the arguments of DiPerna and Lions \cite{DiPernaLions89}, one can show that under these assumptions, the Cauchy problem for the continuity equation \eqref{1} is well-posed in the class of functions~$\rho$ with $\rho \in L^{\infty}((0,T);L^q(\Omega))$. Notice that for the purpose of well-posedness one could drop the assumption that $f\in L^{\infty}((0,T);W^{-1,1}(\Omega))$. This assumption, however, is crucial for the purpose of optimal stability estimates that will enter our analysis. Likewise, in \cite{DiPernaLions89} the vector field $u$ can possibly be unbounded, but for our numerical analysis it is important for $u$ to be in addition uniformly bounded in time and space
\begin{equation}\label{ass:u:infty}
  u \in L^\infty((0,T)\times\Omega) .
\end{equation}

Let us now give our precise result.
\begin{theorem}\label{T1}
Suppose that $\delta_{\max}$, $\rho^0$, $f$, and $u$ are given such that \eqref{43}, \eqref{24}, \eqref{25} and \eqref{ass:u:infty} hold. Let $\rho\in L^{\infty}((0,T);L^q(\Omega))$ denote the exact solution to the continuity equation \eqref{1} and for $\dt\in(0, \dt_{\max}\wedge 1)$ and $h\in (0,1)$ let~$\rho_{\dt,h}$ denote the approximate solution associated via \eqref{29} to the upwind scheme \eqref{scheme} with coefficients \eqref{26}, \eqref{27} and \eqref{28}. Suppose that the mesh is non-degenerate in the sense of~\eqref{3}. Then for any $r>0$ it holds
\begin{equation}
\label{30}
  \inf_{\pi\in\Pi\bra*{\rho(t,\cdot),\rho_{\dt, h}(t,\cdot)}} \iint \log\bra[\bigg]{\frac{|x-y|}r +1} d\pi(x,y) \lesssim 1+ \frac{\sqrt{h\,T\,\|u\|_{L^{\infty}}} + \sqrt{\dt\, T\phantom{\|}\!}\, \|u\|_{L^{\infty}}}r \!
\end{equation}
uniformly in $t\in (0,T)$.
\end{theorem}
Before discussing the result, let us briefly comment on the notation. Here and in the following we write $a\lesssim b$ whenever there is a constant $C$ independent of $h$ and $\dt$ such that $a\le C b$. Notice that we kept the supremum norm of the velocity field and the total time in the right-hand side of \eqref{30} for the purpose of dimensional consistency: Since $r$ has the dimension of a length, the term on the right-hand side is dimension-free. Other terms dependent on the velocity fields, the data or the solution have been absorbed into the implicit constant inside~``$\lesssim$''.

In our main estimate \eqref{30}, $\Pi\bra*{\rho(t,\cdot),\rho_{\dt, h}(t,\cdot)}$ is the set of all joint measures with marginals $(\rho-\rho_{\dt, h} )^+$ and $(\rho-\rho_{\dt, h} )^-$, which is non-empty because $\rho(t,\cdot)$ and $\rhoh(t,\cdot)$ have same total mass, see Equation \eqref{12} on page \pageref{12} below. The quantity on the left-hand side of \eqref{30} is a Kantorovich--Rubinstein distance, that originates from the theory of optimal mass transportation. Roughly speaking, in the original context this distance describes the minimal total cost that is necessary for transferring the configuration $\rho$ into the configuration $\rho_{\dt,h}$ if $\log(z/r+1)$ is the cost for shipping a unit volume over the distance $z$. Notice that the cost function (and thus the Kantorovich--Rubinstein distance) is singular in the limit $r\to0$. We will give a precise definition of the marginal conditions and some important features of Kantorovich--Rubinstein distances in Section \ref{S:KR} below.

The result in Theorem \ref{T1} can be interpreted as follows: Since Kantorovich--Rubin\-stein distances metrize weak convergence (of measures) \cite[Theorem 7.12]{Villani03}, by choosing $r = \sqrt{h} + \sqrt{\dt}$ one deduces from \eqref{30} that
\[
\rho_{\dt, h} \longrightarrow \rho\qquad\text{weakly with rate at most $\sqrt{h} + \sqrt{\dt}$},
\]
as $h\to0$ and $\dt \to0$. The result shows the classical $\nicefrac{1}{2}$ convergence rate for upwind schemes on unstructured meshes found earlier in \cite{Kuznetsov76,JohnsonPitkaranta86,VilaVilledieu03,MerletVovelle07,Merlet07,DelarueLagoutiere11}, just that the strong norms considered in the classical setting are traded for weak convergence measures in the DiPerna--Lions setting. In case of the explicit scheme with Cartesian meshes under a CFL condition of the form $\dt\, \|u\|_{L^{\infty}} \le h$, we obtained an analogous result in our previous work \cite{SchlichtingSeis16a}: We proved that the rate of weak convergence is at least~${\nicefrac12}$. A simple example moreover shows that this bound is optimal! Indeed, considering the one-dimensional setting with constant velocity, we prove that, on the one hand, no convergence rates can be obtained in strong norms: For every small $\eps>0$ there exist an initial configuration such that
\[
\lim_{h\to 0} h^{-\eps} \|\rho-\rho_h\|_{L^1((0,T)\times
 \Omega)} \gtrsim 1.
\]
For these data we show on the other hand that the rate of weak convergence measured in the $1$-Wasserstein distance (i.e., the Kantorovich--Rubinstein distance with linear cost function) is at most $\nicefrac12+\eps$, which almost matches the $\mathcal{O}(h^{\nicefrac12})$ error bound---though for a different weak metric. We believe that similar calculations also apply for the implicit scheme yielding the optimality of Theorem \ref{T1}. We plan to address this issue in the future. Strong convergence without rates of the implicit upwind scheme (and the more general discrete Galerkin approximation) in the setting of this paper was proved earlier by Walkington \cite{Walkington05} and Boyer \cite{Boyer12}.

The order of the upwind scheme is formally $1$. The loss in the convergence rate from~$1$ to~$\nicefrac12$ is caused by numerical diffusion, which is analytically manifested in ``weak $BV$'' estimates (cf.~Proposition~\ref{prop:nabla}). These estimates have in the (heuristic) case $|u_{KL}^n|\sim\|u\|_{L^{\infty}}\sim U$ the form
\[
\|\grad \rho_{\dt,h}\|_{L^1((0,T)\times \Omega)} \lesssim \sqrt{\frac{T}{h\,U}}\qquad\text{and}\qquad \|\partial_t \rho_{\dt,h}\|_{L^1((0,T)\times \Omega)} \lesssim \sqrt{\frac{T}{\dt}}.
\]
Such estimates are standard tools in the convergence analysis of numerical schemes for hyperbolic equations under classical regularity assumptions, see also \cite[Chapters 5--7]{EymardGallouetHerbin00}.

\section{Transport distance with logarithmic cost function}\label{S:KR}

In this section we provide the rigorous  definition of the Kantorovich--Rubinstein distance appearing in our error estimate \eqref{30} and we collect those of its properties which will be relevant in the subsequent analysis. For a general introduction into the theory of optimal transportation, we refer to Villani's  monograph \cite{Villani03}.

Given two nonnegative distributions $\rho_1$ and $\rho_2$ on $\Omega$ with same total mass, i.e.,
\begin{equation}
\label{42}
\io \rho_1\, dx = \io\rho_2\, dx,
\end{equation}
the set $\Pi(\rho_1,\rho_2)$ consists of all those joint measures $\pi$ on the product space $\Omega\times\Omega$ which have the marginals $\rho_1$ and $\rho_2$, that is,
\[
\pi[A\times \Omega ] = \int_A\rho_1\, dx \quad\text{and}\quad \pi[\Omega\times A] = \int_A \rho_2\, dx\qquad\text{for all measurable }A\subset \Omega.
\]
This is equivalent to the requirement that
\begin{equation}
\label{19}
\io \bra[\big]{\zeta(x)  + \theta(y)}\, d\pi(x,y) = \io \zeta \,\rho_1\, dx + \io \theta\,\rho_2\, dx \qquad\text{for all }\zeta,\theta\in C(\overline{\Omega}),
\end{equation}
where $C(\overline{\Omega})$ is the set of all functions that are continuous up to the boundary of  $\Omega$. For any positive number $r$, we then define
\begin{equation}
\label{41}
D_r(\rho_1,\rho_2) : = \inf_{\pi\in\Pi(\rho_1,\rho_2)}\iint \log\bra*{\frac{|x-y|}r+1}\, d\pi(x,y).
\end{equation}
Functionals of this type were originally introduced by Kantorovich to compute the minimal cost for transferring goods from producers to consumers. In this context, $d_r(x,y) := \log({|x-y|}/{r}+1)$ plays the role of a cost function. The measures $\pi$ are usually referred to as transport plans. It is not difficult to see that the infimum in \eqref{41} is in fact attained, see, for instance, Theorem 1.3 in \cite{Villani03}. The corresponding minimizer is unique because $d$ is strictly concave \cite[Theorem 2.45]{Villani03} and will in the sequel be denoted by $\po$ and called \emph{optimal transport plan}.

Instead of working with \eqref{41} directly, we will mostly consider the dual formulation
\begin{equation}
\label{8}
D_r(\rho_1,\rho_2) = \sup_{\zeta} \left\{\io \zeta\, (\rho_1-\rho_2)\, dx:\: |\zeta(x) - \zeta(y)|\le \log\bra*{\frac{\abs{x-y}}{r} + 1} \right\},
\end{equation}
cf.~\cite[Theorem 1.14]{Villani03}, which admits a maximizer $\zo$, the so-called \emph{Kantorovich potential}, cf.~\cite[Exercise 2.35]{Villani03}. In particular, it holds that
\[
D_r(\rho_1,\rho_2) = \iint \log\left(\frac{|x-y|}r +1\right)\, d\pi_{\opt}(x,y) = \io \xi_{\opt}(\rho_1-\rho_2)\, dx.
\]
The duality formula is known as the Kantorovich--Rubinstein theorem and has a number of important implications. First, $D_r(\rho_1,\rho_2)$ is a transshipment cost which only depends on the difference $\rho_1-\rho_2$. We may therefore extend the definition of $D_r(\rho_1,\rho_2)$ to densities of the same mass \eqref{42} that are not necessarily nonnegative, so that
\begin{equation}
\label{50}
D_r(\rho_1,\rho_2) = D_r\bra[\big]{(\rho_1-\rho_2)^+,(\rho_1-\rho_2)^-}.
\end{equation}
Because $d_r(x,y)$ defines a distance on $\Omega$, $D_r(\rho_1,\rho_2)$ becomes a distance on the space of functions with same (finite) mass \cite[Theorem 7.3]{Villani03}. We will accordingly refer to~$D_r(\rho_1,\rho_2)$ as a \emph{Kantorovich--Rubinstein distance}. An immediate consequence of this observation is the validity of the triangle inequality
\begin{equation}
\label{15}
D_r(\rho_1,\rho_2)\le D_r(\rho_1,\rho_3) + D_r(\rho_2,\rho_3).
\end{equation}
There is a second triangle-type inequality that we will make use of later and which immediately follows from the dual formulation \eqref{8}, namely
\begin{equation}
\label{5}
D_r(\rho_1 +  \rho_3,\rho_2 + \rho_4)\le D_r(\rho_1,\rho_2) + D_r( \rho_3, \rho_4).
\end{equation}
By the sublinearity of the logarithm it is $d_r(x,y) \leq \frac{\abs{x-y}}{r}$ and therefore
\begin{equation}\label{e:Dr:W1}
  D_r(\rho_1,\rho_2) \leq \frac{1}{r} \sup_{\zeta} \set*{ \int_\Omega \zeta\, \bra*{ \rho_1 - \rho_2} \, dx : \abs*{\zeta(x) - \zeta(y)} \leq \abs*{x-y}} = \frac{1}{r} \norm{ \rho_1 - \rho_2}_{W^{-1,1}} ,
\end{equation}
where the latter is the $1$-Wasserstein distance in its dual formulation.

One of the fundamental properties of Kantorovich--Rubinstein distances is of topological nature and plays a central role in the interpretation of our main result in Theorem~\ref{T1}: Kantorovich--Rubinstein distances metrize weak convergence (of measures) in the sense that
\[
D_r(\rho_k,\rho)\longrightarrow 0 \qquad\Longleftrightarrow \qquad\io \psi \, \rho_k\, dx\longrightarrow \io \psi \, \rho\, dx \qquad\forall \psi\in C(\overline{\Omega})
\]
as $k\to\infty$, for any sequence $(\rho_k)_{k\in\N}$ of densities of same mass as $\rho$, cf.~\cite[Theorem 7.12]{Villani03}.

We will finally mention a crucial relation between optimal transport plan $\po$ and Kantorovich potential $\zo$. The minimizer $\po$ is concentrated on the set
\[
\set[\big]{(x,y)\in \Omega\times \Omega:\: \zo(x)  - \zo(y) = d_r(x,y)} ,
\]
cf.~\cite[Exercise 2.37]{Villani03}, which in turn yields a formula for the derivative of the Lipschitz function $\zo$. Indeed, it holds that
\begin{equation}
\label{9}
\grad \zo (x)  = \grad\zo(y)= \grad_x d_r(x,y) = \frac1{|x-y|+r}\, \frac{x-y}{|x-y|}
\end{equation}
for $\po$-a.e.\ $(x,y)\in \Omega\times \Omega$. The global estimate 
\begin{equation}
\label{49}\|\grad\zo\|_{L^{\infty}} \le 1/r
\end{equation}
is a consequence of the characteristic Lipschitz condition in \eqref{8}.

\section{The upwind scheme: II. Properties and estimates}\label{S:properties}

Let us start by quoting a result on existence, uniqueness, conservativity and monotonicity of the  upwind scheme, for which we refer to~\cite{EymardGallouetHerbin00}.

\begin{lemma}\label{lem:wellposed}
Under the assumption of Theorem~\ref{T1} has the implicit upwind finite volume scheme~\eqref{scheme} a unique solution. This solution is mass preserving in the sense that
  \begin{equation*}
  \io \rho_{\dt,h}(t,x) \, dx = \io \rho_{\dt,h}^0(x) \, dx .
\end{equation*}
Moreover, if $\rho_h^0$ and $f_{\dt, h}$ are both nonnegative so is $\rho_{\dt, h}$.
\end{lemma}

The next lemma provides numerical stability for the implicit upwind scheme. The numerical stability estimate \eqref{e:Lq:stability} below is the discrete counterpart of the a priori estimate~\eqref{18}. In addition to that we achieve control of spatial and temporal discrete gradients---a manifestation of the numerical diffusion introduced by the scheme---see~\eqref{e:energy_estimate} below. We will see later in Proposition \ref{prop:nabla} that the latter convert into weak $BV$ estimates.

\begin{lemma}[Stability and energy estimate]\label{lem:stability-energy}
  Suppose that $\rho_h^0$ and $f_{\dt,h}$ are nonnegative. Let $\rhoh$ be the solution to the implicit upwind scheme~\eqref{scheme}. Then for any $q\in (1,\infty)$, any $\kappa>1$ and any $\dt \leq \dt_{\max}(\kappa)$ as defined in~\eqref{43}, it holds
  \begin{equation}\label{e:Lq:stability}
     \norm{\rhoh}_{L^\infty(L^q)} \leq \Lambda_{\dt,h}^{\kappa\bra*{1- \frac{1}{q}}} \bra*{ \norm{\rho_h^0}_{L^q} + \norm{f_{\dt,h}}_{L^1(L^q)}} ,
  \end{equation}
  where $\Lambda_{\dt,h}:= \exp\bra[\big]{\norm{ \bra*{ \div u}_{\dt,h}^- }_{L^1(L^\infty)}}$ with $(\div u)_{\dt,h}$ defined analogously to $f_{\dt,h}$ in~\eqref{27} and~\eqref{29}.

  \medskip

  \noindent In addition, for any $\bar q\in (1,\min\set{q,2}]$, the following spatial and temporal discrete gradient bounds hold
   \begin{align}
    \MoveEqLeft{\sum_{n=0}^{N-1} \sum_K |K| \bra*{\frac{\rho_K^{n+1}+\rho_K^n}2}^{\bar q-2} \bra*{\rho_K^{n+1}-\rho_K^n}^2}   \notag \\
     +\, \dt\, \MoveEqLeft{ \sum_{n=0}^{N-1}   \sum_{K} \sum_{L\sim K} \abs{K\edge L} \abs{u_{KL}^n} \bra*{\frac{\rho_K^{n+1}+\rho_L^{n+1}}{2}}^{\bar q-2} \bra*{\rho_K^{n+1} - \rho_L^{n+1}}^2} \label{e:energy_estimate} \\
    &\leq C_{\bar q} \; \Lambda_{\dt,h}^{\kappa\bra*{\bar q-1}}\; \bra*{1+ \norm{ \bra*{ \div u}_{\dt,h}^- }_{L^1(L^\infty)} }\bra*{\|\rho^0_h\|_{L^{\bar q}} + \|f_{\dt,h}\|_{L^1(L^{\bar q})}}^{\bar q}\notag ,
\end{align}
where $C_{\bar q}$ is a numerical factor with $C_{\bar q} \to \infty$ as $\bar q\to 1$.
\end{lemma}

As it will become clear in the proof, the stability estimate \eqref{e:Lq:stability} is also valid in the limiting case $q=1$. However, it is not clear to us how to extend~\eqref{e:energy_estimate} to that case. This, in fact, is the reason why we have to restrict ourselves to the setting with $q>1$.

\begin{proof}
  From the monotonicity of the scheme stated in Lemma~\ref{lem:wellposed} and the nonnegativity assumption on the data $\rho^0_h$ and $f_{\dt,h}$, we deduce that the solution $\rho_{\dt ,h}$ is nonnegative, too. For the proof, it will be convenient to use the formulation \eqref{scheme2} of the upwind scheme. Multiplication of \eqref{scheme2} by $\abs{K}$ yields
  \begin{align*}
    \abs{K} \bra[\big]{\rho^{n+1}_K - \rho^{n}_K} &+ \dt \sum_{L\sim K} \abs{K \edge L} \, u_{KL}^n \, \frac{\rho_K^{n+1}+ \rho_L^{n+1}}{2} \\
    &+ \dt \sum_{L\sim K} \abs{K\edge L} \, \abs*{u_{KL}^n} \, \frac{\rho_K^{n+1} - \rho_L^{n+1}}{2} = \dt \, \abs{K} f_K^n .
  \end{align*}
  Let us denote the terms in the above identity as $\I^n_K + \II^n_K + \III^n_K = \IV^n_K$.
  Our derivation of \eqref{e:Lq:stability} mimics the one of~\eqref{18} in the continuous setting. First, we test the equation with~$(\rho_K^{n+1})^{q-1}$ and sum over $K$. For the first term $\I^n_K$, we use the H\"older inequality to obtain
  \begin{align*}
    \I^n &:= \sum_{K} \I^n_K (\rho_K^{n+1})^{q-1} =   \sum_{K} \abs{K}  \bra*{\rho^{n+1}_K}^q - \sum_{K} \abs{K} \rho^n_K \bra{\rho^{n+1}_K}^{q-1} \\
    &\geq  \norm{\rho_{\dt,h}^{n+1}}_{L^q}^q - \norm{\rho_{\dt,h}^n}_{L^q}  \norm{\rho_{\dt,h}^{n+1}}_{L^q}^{q-1}.
  \end{align*}
  Next, by recalling that $u_{KL}^n = - u_{LK}^n$ we symmetrize the second term $\II^n_K$
  \begin{align*}
    \II^n &:= \sum_{K} \II_K^n (\rho_K^{n+1})^{q-1} \\
    &= \frac{\dt}{2}\, \sum_{K} \sum_{L\sim K} \abs{K\edge L} u_{KL}^n \frac{\rho_K^{n+1} + \rho_L^{n+1}}{2}  \bra*{ (\rho_K^{n+1})^{q-1} - (\rho_L^{n+1})^{q-1} } .
  \end{align*}
  Let us introduce the $q$-mean
  \begin{equation*}
    \theta_q: \R_+ \times \R_+ \to \R_+ \qquad\text{with}\qquad \theta_q(a,b) := \frac{q-1}{q} \frac{a^q - b^q}{a^{q-1} -b^{q-1}}
  \end{equation*}
  and note that $\theta_2(a,b)$ is the arithmetic mean (see Appendix~\ref{s:appendix} for some of its properties). By the definition of the $q$-mean, we have the identity
  \[
    \theta_q(a,b)\bra[\big]{ a^{q-1} - b^{q-1}} = \tfrac{q-1}{q} \bra[\big]{a^q - b^q}.
  \]
  In particular, $\II^n$ can be decomposed into the sum $\II^n_1 + \II^n_2$, where
  \begin{align*}
  \II_1^n &:= \frac{q-1}q\frac{\dt}{2} \sum_K\sum_{L\sim K}  \abs{K\edge L} \, u_{KL}^n \, \bra[\big]{ (\rho_K^{n+1})^{q} - (\rho_L^{n+1})^{q}} ,\\
  \II^n_2 & :=\frac{\dt}2 \sum_K \sum_{L\sim K} \abs{K\edge L}\, u_{KL}^n \, \bra*{\theta_2-\theta_q}\bra[\big]{\rho_K^{n+1} , \rho_L^{n+1}} \; \bra[\big]{ (\rho_K^{n+1})^{q-1} - (\rho_L^{n+1})^{q-1} } .
  \end{align*}
  To  estimate the term~$\II^n_1$, we do another symmetrization and use the divergence theorem in every cell $K$, so that
  \begin{align*}
    \II_1^n &= \dt\, \frac{q-1}{q} \sum_{K}  (\rho_K^{n+1})^{q}  \sum_{L\sim K} \abs{K\edge L} \, u_{KL}^n =  \dt\, \frac{q-1}{q} \sum_{K}  \abs{K} \, (\rho_K^{n+1})^{q} \, \bra*{ \div u}_K^n \\
    &\geq -  \frac{q-1}{q}\lambda^n \norm{\rho_{\dt, h}(t^{n+1})}_{L^q}^q,
  \end{align*}
where we have set  $\lambda^n := \dt\, \|(\div u_{\dt,h}(t^n))^-\|_{L^{\infty}}$ for abbreviation.  Let us now estimate the remainder term $\II^n_2 $. In Appendix~\ref{s:appendix}, we derive the following elementary estimate between $\theta_q$ and $\theta_2$:
  \begin{equation*}
   \qquad \abs*{ \theta_2\bra*{a,b}  - \theta_q\bra*{a  , b}} \leq \frac{\abs{q-2}}{q} \frac{\abs{a-b}}{2}\quad\text{for any } a,b>0 .
  \end{equation*}
  Applying this estimate inside of $\II^n_2$, we arrive at
  \begin{align*}
    \II^n_2 \geq - \frac{\dt}{2} \, \frac{\abs{q-2}}{q} \sum_{K} \sum_{L\sim K}\, \abs{K\edge L} \, \abs*{u_{KL}^n} \frac{\rho_K^{n+1} - \rho_L^{n+1}}{2} \, \bra*{(\rho_K^{n+1})^{q-1} - (\rho_L^{n+1})^{q-1}}
  \end{align*}
Likewise, summation in $K$ and symmetrization leads to a similar bound on $\III^n_K$, namely
  \begin{align*}
    \III^n&:=\sum_{K} \III_K^n (\rho_K^{n+1})^{q-1} \\
    &\geq \frac{\dt}{2}  \sum_{K} \sum_{L\sim K} \abs{K\edge L} \, \abs{u_{KL}^n}  \, \frac{\rho_K^{n+1} - \rho_L^{n+1}}{2} \, \bra*{ (\rho_K^{n+1})^{q-1} - (\rho_L^{n+1})^{q-1} }.
  \end{align*}
    Finally, the term obtained from $\IV^n_K$ after testing by $(\rho_K^{n+1})^{q-1}$ and summation in $K$ is estimated by the Hölder inequality as
  \[
    \IV^n := \sum_{K} \IV_K^n (\rho_K^{n+1})^{q-1} \leq \dt\, \norm{ \rho_{\dt,h}^{n+1}}_{L^q}^{q-1} \norm{ f_{\dt, h}^n}_{L^q}.
  \]
  A combination of all the estimates so far gives
  \begin{equation} \label{e:stability:tested}
  \begin{aligned}
    \MoveEqLeft{\norm[\big]{\rho_{\dt,h}^{n+1}}_{L^q}^q + c_q \, \frac{\dt}{2} \sum_{K} \sum_{L\sim K} \abs{K\edge L} \, \abs{u_{KL}^n}\, \bra[\big]{\rho_K^{n+1} - \rho_L^{n+1}}\bra[\big]{ (\rho_K^{n+1})^{q-1} - (\rho_L^{n+1})^{q-1} }}  \\
    &\leq \norm[\big]{\rho_{\dt,h}^n}_{L^q} \, \norm[\big]{\rho_{\dt,h}^{n+1}}_{L^q}^{q-1} +  \tfrac{q-1}{q}
    \, \lambda^n\, \norm[\big]{\rho_{\dt, h}^{n+1}}_{L^q}^q + \dt \, \norm[\big]{ \rho_{\dt,h}^{n+1}}_{L^q}^{q-1} \, \norm[\big]{ f_{\dt, h}^n}_{L^q} ,
\end{aligned}  \end{equation}
   where $c_q$ is obtained as
  \[
    \frac{1}{2} \bra*{ 1 - \frac{\abs{q-2}}{q}} = \min\set*{\frac{q-1}{q},\frac{1}{q}}  =:  c_q .
  \]
  Dropping for the moment the second term on the left hand side of~\eqref{e:stability:tested} and dividing by $\norm{ \rho_{\dt,h}^{n+1}}_{L^q}^{q-1}$ gives the bound
  \begin{equation}\label{e:stability:Gronwall}
    \norm{\rho_{\dt,h}^{n+1}}_{L^q} \bra[\big]{ 1 -  \tfrac{q-1}{q}  \lambda^n} \leq \norm{\rho_{\dt,h}^{n}}_{L^q} + \dt  \norm{ f_{\dt, h}^n}_{L^q}.
  \end{equation}
  From the choice of $\dt_{\max}(\kappa)$ in~\eqref{43} it follows $\, \tfrac{q-1}{q} \lambda^n \leq \tfrac{\kappa-1}{\kappa}$ and therefore
  \[
    \frac{1}{1- \, \tfrac{q-1}{q} \lambda^n} \leq 1 + \kappa \, \tfrac{q-1}{q}\,\lambda^n \leq \exp\bra[\big]{ \kappa \, \tfrac{q-1}{q}\lambda^n} .
  \]
An iteration of \eqref{e:stability:Gronwall} thus generates the bound \eqref{e:Lq:stability}.

In order to establish the energy estimates, we reconsider term $\I^n$. By the convexity of the map $a\mapsto a^{\bar q}$, we notice that for $\bar q \in (1,2]$ it holds
\begin{align*}
  \bar q \, a^{\bar q-1}(a-b) &\ge  a^{\bar q} -b^{\bar q} +\frac{\bar q(\bar q-1)}2 \min\{a^{\bar q-2},b^{\bar q-2}\} (a-b)^2 \\
  &\geq  a^{\bar q} -b^{\bar q} + \frac{\bar q(\bar q-1)}{2^{3-\bar q}} \bra*{ \frac{a+b}{2}}^{\bar q -2} (a-b)^2 .
\end{align*}
Applying this estimate with $a=\rho_K^{n+1}$ and $b = \rho_K^n$ yields
\[
\I^n \ge \frac1{\bar q} \sum_K |K|\bra*{(\rho^{n+1}_K)^{\bar q} - (\rho_K^n)^{\bar q}} + \frac{\bar q -1}{2^{3-\bar q}} \sum_K |K| \bra*{\frac{\rho_K^{n+1} + \rho_K^n}{2}}^{\bar q-2} \bra[\big]{\rho_K^{n+1}- \rho_K^n}^2.
\]
If we use this bound on $\I^n$, we obtain instead of \eqref{e:stability:tested} that
\begin{align*}
  &\tfrac1{\bar q} \norm[\big]{\rho_{\dt,h}(t^{n+1})}_{L^{\bar q}}^{\bar q}
    + \frac{\bar q -1}{2^{3-\bar q}}\sum_K |K| \bra[\bigg]{\frac{\rho_K^{n+1} + \rho_K^n}{2}}^{\bar q-2} \bra[\big]{\rho_K^{n+1}- \rho_K^n }^2  \\
    &\qquad +  c_{\bar q} \, \frac{\dt}{2} \sum_{K} \sum_{L\sim K} \abs{K\edge L} \abs{u_{KL}^n}  \bra[\big]{\rho_K^{n+1} - \rho_L^{n+1}}\bra*{ \bra[\big]{\rho_K^{n+1}}^{\bar q-1} - \bra[\big]{\rho_L^{n+1}}^{\bar q-1} } \\
  \leq\ &\tfrac{1}{\bar q} \norm[\big]{\rho_{\dt,h}(t^{n})}_{L^{\bar q}}^{\bar q} +  \tfrac{\bar q-1}{\bar q} \,
    \lambda^n \norm[\big]{\rho_{\dt, h}(t^{n+1})}_{L^{\bar q}}^{\bar q} + \dt \norm[\big]{ \rho_{\dt,h}(t^{n+1})}_{L^{\bar q}}^{\bar q-1} \norm[\big]{ f_{\dt, h}(t^n)}_{L^{\bar q}} .
\end{align*}
It is   furthermore convenient to rewrite the third term on the left-hand side 
by applying  
  the following elementary inequality for any $\bar q\in (1,2]$ and any $a,b>0$
  \[
    \bra*{a-b}^2 \bra*{ \frac{a+b}{2}}^{\bar q-2} \leq \bra*{a-b} \frac{a^{\bar q-1} - b^{\bar q-1}}{\bar q-1} ,
  \]
  with $a=\rho_K^{n+1}$ and $b=\rho_L^{n+1}$.
  Therewith, summation over $n$, then yields 
  \begin{align*}
    &\sum_K |K| \bra*{\frac{\rho_K^{n+1}+\rho_K^n}2}^{\bar q-2} \bra*{\rho_K^{n+1}-\rho_K^n}^2  \\
     +  \dt &\sum_{K} \sum_{L\sim K} \abs{K\edge L} \abs{u_{KL}^n}  \bra*{\frac{\rho_K^{n+1}+\rho_L^{n+1}}{2}}^{\bar q-2} \bra[\big]{\rho_K^{n+1} - \rho_L^{n+1}}^2 \\
    \leq C_{\bar q} &\bra*{ \norm[\big]{\rho^0_h}_{L^{\bar q}}^{\bar q} \! +  \norm[\big]{(\div u)^-_{\dt,h}}_{L^1(L^{\infty})} \norm[\big]{\rho_{\dt, h} }_{L^{\infty}(L^{\bar q})}^{\bar q} \! + \norm[\big]{ \rho_{\dt,h} }_{L^{\infty}(L^{\bar q})}^{{\bar q}-1} \norm[\big]{ f_{\dt, h}}_{L^1(L^{\bar q})}} .
      \end{align*}
 Applying the just proven stability estimate \eqref{e:Lq:stability}  yields \eqref{e:energy_estimate}.
 \end{proof}

Before stating and proving the afore mentioned weak $BV$ estimates for the upwind scheme, 
we note two obvious relations between exact and approximated data.
From the definitions in Section \ref{S:scheme}, it immediately follows that the initial total masses are identical in the continuous and the discrete models:
\begin{equation*}
\io \rho_h^0\, dx = \sum_{K\in\cT} \int_K \rho_h^0\, dx = \sum_{K\in\cT}\int_K \rho^0\, dx = \io \rho^0\, dx .
\end{equation*}
Similarly, as for the continuous problem, the approximate net source is vanishing at any time $t\in(0,T)$, because with $n\in\llb 0,N-1\rrb$ such that $t\in [t^n, t^{n+1})$ it holds
\begin{equation*}
\io f_{\dt,h}(t,x)\, dx = \sum_{K\in\cT} \int_K f_{\dt,h}(t,x)\, dx = \frac1{\dt}\int_{t^n}^{t^{n+1}} \io f(t,x)\, dx \, dt\overset{\eqref{11}}{ =} 0 .
\end{equation*}
We will see in Lemma~\ref{lem:wellposed} below that these two facts together entail for any $t\in (0,T)$ that
\begin{equation}
\label{12}
\io \rhoh(t,x)\, dx = \io \rho(t,x)\, dx .
\end{equation}

Let us now prove some basic estimates between  discretized and continuous versions of various quantities.
\begin{lemma}\label{lem:disc:norm}
The following estimates hold: $\|f_{\dt,h}\|_{L^1(L^q)}\le \|f\|_{L^1(L^q)}$, $\|\rho_h^0\|_{L^q}\le \|\rho^0\|_{L^q}$ and
$
\norm{\bra{\div u}_{\dt,h}^-}_{L^1(L^\infty)} \leq \norm{\bra*{\div u}^-}_{L^1(L^\infty)} .
$
\end{lemma}

\begin{proof}
The estimate for the initial data is a straight-forward consequence of Jensen's inequality. Indeed,
\[
\|\rho_h^0\|_{L^q}^q  = \sum_{K\in\cT} |K| \left|\avint_K \rho^0\, dx\right|^q \le \sum_{K\in\cT} \int_K |\rho^0|^q\, dx = \|\rho^0\|_{L^q}^q.
\]
By essentially the same reasoning, for any source term of the form $f(t,x) = f^1(t)f^2(x)$, it holds that
\[
\|f_{\dt,h}\|_{L^1(L^q)} = \|f^1_{\dt}\|_{L^1}\|f^2_h\|_{L^q} \le \|f^1\|_{L^1}\|f^2\|_{L^q} = \|f\|_{L^1(L^q)},
\]
if by subscript $\dt$ and $h$ we denote the discretization in time and space, respectively. It remains to conclude with an approximation argument: Thanks to the density of smooth functions in $L^1(L^q)$, it is enough to prove the statement for continuous functions. Moreover, the Stone--Weierstra\ss\ theorem enables us to furthermore approximate continuous functions uniformly by functions of the form $f^1(t)f^2(x)$---for which the estimate is shown above. The proof for $\div u$ follows from a combination of the Jensen inequality applied to the convex function $a\mapsto (a)^-$ with the argument above.
\end{proof}

As a consequence, under the assumptions of Theorem \ref{T1}, the expressions in the right-hand sides of \eqref{e:Lq:stability} and \eqref{e:energy_estimate} are both $\mathcal{O}(1)$, for instance,
\begin{equation}
\label{47}
\|\rho_{\dt,h}\|_{L^{\infty}(L^q)} \lesssim1.
\end{equation}

Let us now establish the  weak $BV$ estimates in space and time, which will occur in later estimates and are a manifestation of the numerical diffusion.
\begin{proposition}[Weak $BV$ estimates]\label{prop:nabla}
  Suppose the assumptions of Theorem~\ref{T1} hold and let $\rhoh$ be the solution to the implicit upwind scheme~\eqref{scheme}. Then, it holds
  \begin{align}
    \sum_{n=0}^{N-1} \sum_{K} \abs{K}  \abs{ \rho_K^{n+1} - \rho_K^n} &\lesssim \sqrt{\frac{T}{\dt}} , \label{est:nabla:2}\\
    \dt \sum_{n=0}^{N-1} \sum_K\sum_{L\sim K} |K\edge L||u_{KL}^n| |\rho_K^{n+1}-\rho_L^{n+1}| &\lesssim \sqrt{\frac{T \norm{u}_{L^\infty}}{h}} . \label{est:nabla:1}
  \end{align}
\end{proposition}
 Notice that these estimates do not have any counterparts in the continuous model \eqref{1}. In particular, no compactness estimates can be inferred.
\begin{proof}
  The solution to the upwind scheme can be split into $\rho_{\dt,h} = (\rho_{\dt, h})_+ - (\rho_{\dt, h})_-$ where $(\rho_{\dt,h})_{\pm}$ is the nonnegative discrete solution with data $(\rho_h^0)^{\pm}$ and $(f_{\dt,h})^{\pm}$ (first discretized then decomposed). Hence, once we have established the estimates~\eqref{est:nabla:1} and~\eqref{est:nabla:2} for nonnegative data, the estimate follows for general data just by the triangle inequality and the above observation. Therefore, let $\rho_{\dt,h}$ be the nonnegative solution of the upwind scheme to the nonnegative data $\rho_h^0$ and $f_{\dt,h}$. The term~\eqref{est:nabla:1} is estimated by applying the Cauchy-Schwarz inequality, whereby we smuggle in an additional weight $\bra*{ (\rho_K^{n+1} + \rho_L^{n+1})/2}^{\bar q-2}$ for some $\bar q \in (1,\min\set{q,2}]$:
  \begin{align*}
    \MoveEqLeft{\sum_K\sum_{L\sim K} |K\edge L||u_{KL}^n| |\rho_K^{n+1}-\rho_L^{n+1}|} \\
    &\leq \bra*{\sum_K \sum_{L\sim K} \abs{K \edge L} \abs{u_{KL}^n} \bra*{ \rho_K^{n+1} - \rho_L^{n+1}}^2 \bra*{ \frac{\rho_K^{n+1} + \rho_L^{n+1}}{2}}^{\bar q-2}}^{\frac{1}{2}}  \\
    &\qquad \times \bra*{\sum_K \sum_{L\sim K} \abs{K \edge L} \abs{u_{KL}^n}  \bra*{ \frac{\rho_K^{n+1} + \rho_L^{n+1}}{2}}^{2-\bar q}}^{\frac{1}{2}}  =: \bra*{\I^n \times \II^n}^{\frac{1}{2}}.
  \end{align*}
 After summing over $n$ and another H\"older inequality in time,  the term $\I^n$ can be directly estimated by the energy estimate~\eqref{e:energy_estimate} from Lemma~\ref{lem:stability-energy} for $\bar q\in (1,\min\set{q,2}]$.
  For the term $\II^n$, we observe that $\bra*{(a+b)/2}^{2-\bar q} \leq a^{2-\bar q} + b^{2-\bar q}$ for any $a,b> 0$ and use the regularity assumption on the mesh~\eqref{e:iso:mesh} to estimate
  \begin{align*}
    \II^n \leq 2\|u\|_{L^{\infty}}\sum_K    \bra*{\rho_K^{n+1}}^{2-\bar q} \sum_{L\sim K} \abs{K\edge L} \lesssim \frac{\|u\|_{L^{\infty}}}{h} \sum_{K} \abs{K} \bra*{\rho_K^{n+1}}^{2-\bar q} .
    \end{align*}
  Now, we choose $\bar q = q$ if $q\in (1,2)$ and obtain after summation over $n$:
  \begin{align*}
    \sum_{n=0}^{N-1}\dt \,\II^n \lesssim \frac{T \|u\|_{L^{\infty}}}h \|\rho_{\dt,h}\|_{L^{\infty}(L^{2-q})}.
  \end{align*}
  Since $q\geq 1$, we have that $2-q < q$ and can estimate $\|\rho_{\dt,h}\|_{L^{\infty}(L^{2-q})}$ by $\|\rho_{\dt,h}\|_{L^{\infty}(L^{q})}$ up to a factor depending on $\abs{\Omega}$.
  In the case $q\geq 2$, we choose $\bar q = 2$ in \eqref{e:energy_estimate} and deduce the analogous result. Notice that the solution is uniformly bounded in  any $L^{\infty}(L^{\bar q})$ with $\bar q\le q$ thanks to the stability estimate \eqref{e:Lq:stability} and the bounds from Lemma \ref{lem:disc:norm}. A combination of the previous estimates yields \eqref{est:nabla:1}.
  
  The estimate~\eqref{est:nabla:2} follows along the same lines by first applying the Cauchy-Schwarz inequality with the same weight and then the a priori estimates \eqref{e:Lq:stability} and \eqref{e:energy_estimate}.
\end{proof}

\section{Error estimates and proof of Theorem~\ref{T1}}\label{S:estimates}
In this section, we present the proof of our main result Theorem~\ref{T1}. We will see that there are two classes of  discretization errors  contributing to  estimate \eqref{30}. The first class consists of errors introduced by the discretization of time and space, and thus by the corresponding finite volume approximations of data and coefficients. These errors are  $\mathcal{O}(h+\dt)$. The second class is caused by the discretization of the scheme, which is usually referred to as the truncation error. These latter errors are all $\mathcal{O}(h^{\nicefrac{1}{2}}+\dt^{\nicefrac{1}{2}})$, and are related to the phenomenon of numerical diffusion, see, for instance, Section~2.4 in~\cite{SchlichtingSeis16a}.

\subsection{Discretization of data}

We begin the error analysis by addressing the various errors caused by the discretization of time and  space.  The first one concerns the discretization of the time steps for the continuous problem.
\begin{lemma}[Discretization of time]\label{L1}
For $n\in \llb 0,N-1\rrb$ and any $t\in [t^n, t^{n+1})$ it holds
\begin{equation*}
  D_r\bra[\big]{\rho(t),\rho(t^n)} \lesssim \frac{\dt \norm{u}_{L^\infty}}{r}  + \frac{\dt\, \|f\|_{L^{\infty}(W^{-1,1})}}{r}.
\end{equation*}
\end{lemma}
The next error is caused by the spatial discretization of the data. Here $f_h$ is defined analogously to $\rho_h^0$, thus still continuous in time. The statement of this lemma is an immediate consequence of the stability estimates for continuity equations established in~\cite{Seis16a}. The argument for the initial data is already given in~\cite[Lemma 8]{SchlichtingSeis16a}.
\begin{lemma}[Spatial discretization of data]\label{L2}
Let $\rho^h$ be the solution of the continuity equation~\eqref{1} with initial datum $\rho^0_h$ and  source-sink distribution $f_{h}$. Then it holds for any $t\in [0,T]$ that
\begin{equation*}
  D_r\bra[\big]{\rho(t),\rho^h(t)} \lesssim 1 + \frac{h }{r} .
\end{equation*}
\end{lemma}
Thirdly, we consider the error caused by discretizing  data and coefficients in the time variable. Here the subscript $\dt$ refers to discretization by averaging over $\bigl[t^n,t^{n+1}\bigr)$.
\begin{lemma}[Temporal discretization of data]\label{L3}
Let $\rho^{\dt}$ be the solution of the continuity equation~\eqref{1} with driving vector field $u_{\dt}$ and source-sink distribution~$f_{\dt}$. Then it holds for any $\ell\in\llb 0,N\rrb$ that
\begin{equation*}
  D_r\bra[\big]{\rho(t^{\ell}),\rho^{\dt}(t^{\ell})} \lesssim 1 + \frac{\dt \norm{u}_{L^\infty}}{r} .
\end{equation*}
\end{lemma}
The argument for this last error estimate substantially differs from those for Lemmas~\ref{L1} and~\ref{L2} (and those from \cite{Seis16b,SchlichtingSeis16a}). In fact, to control the errors caused by the discretization of time we have to subtly change from Eulerian to Lagrangian coordinates.

We postpone the proofs of Lemmas \ref{L1} to \ref{L3} until Subsection \ref{SS:proofs} below. In view of the triangle inequality \eqref{15}, these first results imply that for any $t\in (0,T) $ it holds
\begin{equation}
\label{48}
D_r(\rho(t),\rho_{\delta, h}(t)) \lesssim D_r(\rho^{\delta, h}(t^{\ell}),\rho_{\delta, h}(t^{\ell})) + 1 + \frac{\delta\, \|u\|_{L^{\infty}}}r +\frac{\dt\, \|f\|_{L^{\infty}(W^{-1,1})} }r  + \frac{h}r,
\end{equation}
where $\ell\in \llb 0,N\rrb$ is such that $t\in \left[t^{\ell},t^{\ell+1}\right)$ and $\rho^{\delta,h}$ is the unique distributional solution to the continuity equation \eqref{1} with initial datum $\rho_h^0$, source-sink distribution $f_{\delta,h}$ and velocity $u_{\delta}$.  The remaining error $D_r(\rho^{\delta, h}(t^{\ell}),\rho_{\delta, h}(t^{\ell}))$ governs the convergence rate. Its treatment will be illustrated in what follows.

\subsection{The error caused by the scheme}

By the virtue of eliminating the discretization errors in the previous subsection, we tacitly assume from here on that the continuous problem is solved with data $\rho_h^0$, $f_{\dt,h}$ and $u_{\dt}$ and that $t = t^{\ell}$ for some $\ell\in \llb0,N\rrb$, i.e., we assume that $\rho(t)  =\rho^{\delta ,h}(t^{\ell})$.  We will see that the truncation error caused by the scheme is $\mathcal{O}(\sqrt{h}+\sqrt{\dt})$.
Our goal in the following is to portray the main steps in the estimate of the discrete rate of change of the distance $D_r(\rho(t^n),\rho_{\delta, h}(t^n))$.
In fact, it turns out that instead of analyzing $\rho_{\delta,h}$, it is more convenient to study the piecewise linear temporal approximate solution defined by
\[
\widehat \rho_{\dt,h}(t,x) := \frac{t-t^n}{\dt} \rho^{n+1}_K + \frac{t^{n+1}-t}{\dt} \rho^n_K\quad\text{for a.e.\ }(t,x)\in \bigl[t^n,t^{n+1}\bigr)\times K.
\]
Since $\widehat \rho_{\dt,h}(t^n) = \rho_{\dt,h}(t^n)$ for any $n\in \llb0,N\rrb$, there is no additional  error term to consider by replacing the piecewise constant (in time) approximation by the piecewise linear approximation. The advantage of considering $\widehat \rho_{\dt,h}$ instead of $\rho_{\dt,h}$ is that the former is (weakly) differentiable.  Indeed, in a first step, we formally compute the rate of change of the Kantorovich--Rubin\-stein distance between $\rho$ and $\widehat \rho_{\dt,h}$:
\[
\frac{d}{dt} D_r(\rho,\widehat \rho_{\dt, h}) = \io \zo (\partial_t\rho-\partial_t\widehat \rho_{\dt,h})\, dx.
\]
Here, $\zo = \zo(t)$ denotes the Kantorovich potential corresponding to $D_r(\rho,\widehat \rho_{\dt, h}) $ at  time $t$. By construction, it holds $\partial_t\widehat \rho_{\dt,h} = \dt^{-1} (\rho_h^{n+1} - \rho_h^n)$ where we have set $\rho_h^n = \rho_h(t^n)$ for any $n$. The time derivative of $\rho$ does in general not exist. However, using the continuity equation and formally integrating by parts, the above formula may be rewritten as
\[
\frac{d}{dt}D_r(\rho,\widehat \rho_{\dt, h}) = \io \grad\zo\cdot u \, \rho\, dx  + \io \zo f\, dx - \frac1{\dt}\io \zo \bra*{\rho_h^{n+1}  - \rho_h^n}\, dx.
\]
In fact, arguing as in  \cite[Lemma 1]{Seis16a}, this identity can be established rigorously. We now use formulation \eqref{scheme2} of the upwind scheme and notice that the source term drops out because $f=f_{\dt,h}$ by the virtue of Lemma \ref{L2}. After integration over $\bigl[t^n,t^{n+1}\bigr)$, we thus find
\begin{equation}\label{46}
D_r\bra[\big]{\rho(t^{n+1}),\rho_{\dt, h}(t^{n+1})} - D_r\bra[\big]{\rho(t^n),\rho_{\dt ,h}(t^n)} = \I^n + \II^n + \III^n  ,
\end{equation}
where
\begin{align*}
\I^n &:= \int_{t^n}^{t^{n+1}} \io \grad \zo \cdot u\bra[\big]{\rho-  \rho_h^{n+1}}\, dx\,dt
,\\
\II^n&:= \int_{t^n}^{t^{n+1}} \io \grad \zo\cdot u \, \rho_h^{n+1} \, dx\,dt + \dt \sum_K (\zo)^n_K\sum_{L\sim K} \abs{K\edge L} u_{KL}^n \frac{\rho_K^{n+1} + \rho_L^{n+1}}{2}
,\\
\III^n&:= \dt \sum_K(\zo)^n_K \sum_{L\sim K} \abs{K\edge L}|u_{KL}^n| \frac{\rho_K^{n+1} - \rho_L^{n+1}}{2},
\end{align*}
with $(\zo)_K^n := \avint_{t^n}^{t^{n+1}}\avint_K \zo\, dx\,dt$.

Up to a shift in the time variable, the first error term $\I^n$ can be controlled by the techniques developed in \cite{Seis16a} to establish stability estimates for continuity equations.  The time shift can then be compensated with the help of   the temporal weak $BV$ estimate~\eqref{est:nabla:2} from Proposition~\ref{prop:nabla}.
\begin{lemma}[Estimate of $\I^n$] \label{L4}
\[
\sum_{n=0}^{N-1} \I^n \lesssim 1 + \frac{\sqrt{\dt\, T}}r\|u\|_{L^{\infty}} .
\]
\end{lemma}
The second error term $\II^n$ has to be rewritten using an integration by parts. An important ingredient in the following bound is  the spatial weak $BV$ estimate~\eqref{est:nabla:1} from Proposition~\ref{prop:nabla}.
\begin{lemma}[Estimate of $\II^n$]\label{L6}
\[
\sum_{n=0}^{N-1} \II^n \lesssim \frac{\sqrt{h T \norm{u}_{L^\infty}} + h}{r}.
\]
\end{lemma}
The last error term $\III^n$ is of purely diffusive origin
 and is controlled again thanks to~\eqref{est:nabla:1} from Proposition~\ref{prop:nabla}.
\begin{lemma}[Estimate of $\III^n$]\label{L7}
\[
\sum_{n=0}^{N-1} \III^n \lesssim \frac{\sqrt{h T \norm{u}_{L^\infty}}}{r} .
\]
\end{lemma}

The proof of Theorem~\ref{T1} now follows by a combination of the Lemmas~\ref{L1}--\ref{L7}. Indeed, summing over $n\in \llb0,\ell-1\rrb$ in \eqref{46}, using the fact that $\rho(0) = \rho^{\dt,h}(0) = \rho_{\dt,h}(0)$ by the hypothesis of this subsection, and using Lemmas \ref{L4}--\ref{L7}, we find that
\begin{align*}
D_r(\rho(t^{\ell}),\rho_{\dt,h}(t^{\ell})) & = \sum_{n=0}^{\ell-1} \left( D_r(\rho(t^{n+1}),\rho_{\dt,h}(\rho^{n+1})) - D_r(\rho(t^{n}),\rho_{\dt,h}(\rho^{n}))\right) \\
&\lesssim 1 + \frac{\sqrt{\dt T}\, \|u\|_{L^{\infty}}}r + \frac{\sqrt{h T\, \|u\|_{L^{\infty}}}}r
\end{align*}
if $h$ is sufficiently small. Plugging this estimate into \eqref{48} then yields the statement in Theorem \ref{T1}.

\subsection{The proofs of the Lemmas \ref{L1}--\ref{L7}}\label{SS:proofs}

To simplify the notation  in the following, we occasionally write $\psi_t$ for $\psi(t,\cdot)$ for a given function $\psi\in L^1((0,T);X)$. 

We start with the short proof of Lemma \ref{L1}.

\begin{proof}[Proof of Lemma \ref{L1}]
We choose $n\in\llb 0,N\rrb$ such that $t\in [t^n,t^{n+1})$ and write $\zeta_t = \zo(t,\cdot)$ if $\zo$ is the Kantorovich potential corresponding to $D_r(\rho_t,\rho_{t^n})$, so that $D_r(\rho_t,\rho_{t^n}) = \io\zeta_t(\rho_t-\rho_{t^n})\, dx$. Using the distributional formulation of \eqref{1}, we then find
\[
D_r(\rho_t,\rho_{t^n}) = \int_{t^n}^t \io \grad\zeta_t\cdot u_s \, \rho_s\, dx\,ds + \int_{t^n}^t\io \zeta_tf_s\, dx\,ds.
\]
In view of \eqref{49}, the first term on the right-hand side is easily controlled by
\[
\frac{\dt}r \|u\|_{L^{\infty}} \|\rho\|_{L^{\infty}(L^1)} \stackrel{\eqref{18}}{\lesssim} \frac{\dt}r \|u\|_{L^{\infty}}.
\]
For the second term we have by using \eqref{e:Dr:W1}
\[
 \int_{t^n}^t\io \zeta_tf_s\, dx\,ds\le \frac1r \int_{t^n}^{t} \|f\|_{W^{-1,1}}\, ds \le \frac{\dt}r\|f\|_{L^{\infty}(W^{-1,1})}.
 \]
 A combination of the previous estimates yields the statement of the lemma.
\end{proof}

The proofs of Lemmas \ref{L2}--\ref{L7} have the flavor of  stability estimates for continuity equations as those that were recently established in \cite{Seis16a,Seis16b}. Their main ingredients are new estimates on the rate of change of the Kantorovich--Rubinstein distance in which Crippa--De Lellis-type inequalities play a central role. Such inequalities, first derived in \cite{CrippaDeLellis08} and adapted to the context of Kantorovich--Rubinstein distances in \cite{BOS} and \cite{OSS},  provide a way of bounding integrals of difference quotients  by $L^p$ norms of gradients.  At the heart of their proofs is a standard tool from harmonic analysis: the maximal function. The maximal  function of a given function $v $ on $\Omega$ is defined by
\[
Mv(x) := \sup_{r>0}\frac1{r^d} \int_{B_r(x)\cap \Omega} |v|\, dx.
\]
One of the fundamental estimates in the theory of singular integrals is the $L^p$ bound on the maximal function,
\begin{equation}
\label{6}
\| M v\|_{L^p} \lesssim \|v\|_{L^p},
\end{equation}
which is valid for any $p\in(1,\infty]$. In the sequel, we will apply this estimate to the velocity gradient $\grad u$. The fact that \eqref{6} fails in the case $p=1$ is the reason why or theory does not extend to the setting where $u\in L^1(W^{1,1})$. Moreover, we need a pointwise estimate which is of Morrey-type, namely
\begin{equation}
\label{7}
\frac{|v(x) - v(y)|}{|x-y|} \lesssim (M \grad \bar v)(x) + (M\grad \bar v)(y) \qquad\text{ for a.e. } x,y \in \Omega .
\end{equation}
Hereby, $\bar v$ denotes a Sobolev extension of $v$ on $\R^d$. The fundamental  inequality for maximal functions \eqref{6} can be found in any standard reference on harmonic analysis, see, for instance, \cite[p.\ 5, Theorem 1]{Stein70}. The Morrey-type estimate \eqref{7} is rather elementary. Its proof is, for instance, contained in \cite[p.\ 143, Theorem 3]{EvansGariepy92}. The Sobolev extension theorem is given in \cite[p.\ 135, Theorem 1]{EvansGariepy92}.

\begin{proof}[Proof of Lemma \ref{L2}]
For notational convenience in the proof, we omit the time dependence of the occurring functions, if the context allows it. We first notice that the  rate of change of the Kantorovich distance between the two solutions $\rho$ and $\rho^h$ which are both advected by the same velocity field takes the form
\[
\frac{d}{dt} D_r\bra[\big]{\rho,\rho^h} = \io \grad \zo\cdot u\;\bra*{\rho-\rho^h}\, dx + \io \zo \, (f-f_{h})\, dx.
\]
The proof for the homogeneous continuity equation is contained in \cite{Seis16a} and is easily adapted to the case that includes sources and sinks.
The second term on the right-hand side is estimated by $D_r(f,f_{h})$ due to the dual variational formula~\eqref{8}. For the first term we claim that
\begin{equation}
\label{16}
\left| \io \grad \zo\cdot u\;\bra[\big]{\rho-\rho^h}\, dx \right| \lesssim \| u\|_{W^{1,p}}\|\rho-\rho^h\|_{L^{\infty}(L^q)}.
\end{equation}
We postpone the proof of \eqref{16} until  later and continue with the estimate for $D_r\bra*{\rho,\rho^h}$. Integration in time and the a priori estimate \eqref{18} imply that 
\begin{align*}
\sup_{(0,T)} D_r\bra[\big]{\rho,\rho^h}
 &\le D_r\bra[\big]{\rho^0,\rho_h^0} + \int_0^T D_r(f,f_{h}) \, dt\\
 &\quad + C\Lambda^{1-\frac1q}\|u\|_{L^1(W^{1,p})}\bra[\big]{\|\rho^0\|_{ L^q} + \|f\|_{L^1(L^q)}}
\end{align*}
for some $C<\infty$. The first two terms on the right-hand side are controlled in a similar way. It is thus enough to focus on one of them, say $D(\rho^0,\rho_h^0)$. The estimate of this term was already given in  \cite[Lemma 8]{SchlichtingSeis16a}. For the convenience of the reader, we provide the proof here again; this time, however, we present an  argument based on \eqref{8}. If $\zo$ denotes the associated Kantorovich potential, it holds
\[
D(\rho^0,\rho_h^0)  =\io \zo \bra[\big]{\rho^0-\rho^0_h}\, dx = \io \bra[\big]{\zo - (\zo)_h}\rho^0\, dx,
\]
where in the second equality we have used the symmetry of the averaging operator $(\cdot)_h$ defined by averaging over each cell $K\in\cT$. For every $x\in K$ we use the Lipschitz property of $\zo$, cf.~\eqref{8}, in the sense that
\[
\left|\zo(x) - (\zo)_h(x)\right| \le \avint_K \left| \zo(x) - \zo(y)\right|\, dy \le \log\bra*{\frac{h}r+1}\le\frac{h}r.
\]
It immediately follows that $D(\rho^0,\rho_h^0) \le {h}/{r} \|\rho^0\|_{L^1}$.

It remains to prove \eqref{16}. We use the marginal condition \eqref{19} and the formula \eqref{9} for the gradient of the Kantorovich potential to rewrite
\[
\int \grad \zo \cdot u\;\bra*{\rho-\rho^h}\, dx = \iint\frac1{|x-y|+r} \frac{x-y}{|x-y|}\cdot  \bra[\big]{u(x) - u(y)}\, d\po(x,y).
\]
Here $\po$ is the optimal transport plan that corresponds to the marginals $(\rho-\rho^h)^+$ and $(\rho-\rho_h)^-$, cf.~\eqref{50}.
At this point, we need an estimate of Crippa--De Lellis-type to control the expression on the right-hand side by $\|\grad u\|_{L^p}$. Dropping $r$ and making use of the Morrey-type inequality \eqref{7}, we find that
\[
\left|\int \grad \zo\cdot u\;\bra*{\rho-\rho^h}\, dx\right| \lesssim \iint \bra[\big]{(M\grad \bar u)(x) + (M\grad \bar u)(y)}\, d\po(x,y),
\]
where  $\bar u$ denotes a Sobolev extension of $u$ to all of $\R^d$. In view of the marginal condition~\eqref{19}, the term on the right-hand side becomes
\[
\iint \bra[\big]{(M\grad \bar u)(x) + (M\grad \bar u)(y)}\, d\po(x,y)= \int M\grad  \bar u\,  | \rho - \rho^h|\, dx .
\]
An application of H\"older's inequality and of the fundamental estimate \eqref{16} thus yield
\[
\left|\int \grad \zo\cdot u\; \bra*{\rho-\rho^h}\, dx\right| \lesssim \|\grad \bar u\|_{L^p} \|\rho - \rho^h\|_{L^q}\le\|\grad \bar u\|_{L^p} \left(\|\rho\|_{L^{\infty}(L^q)} +\| \rho^h\|_{L^{\infty}(L^q)} \right).
\]
The statement in \eqref{16} now follows from the continuity of the extension operator, the a priori estimate \eqref{18} for the continuity equation and the Young-type estimates in Lemma~\ref{lem:disc:norm}.
\end{proof}

The proof of Lemma~\ref{L3} relies in addition on the superposition principle, which represents the solution of the continuity equation~\eqref{1} in terms of a  Lagrangian flow, and some properties of this flow. We will see that it is enough to consider the classical setting with spatially smooth velocity fields. 
For further references and the convenience of the reader, we recall its definition and basic properties. For $0\leq s \leq t \leq T$ a mapping $\Phi_{t,s} : \Omega \to \Omega$ is called \emph{Lagrangian flow}  if
 for every $x\in \Omega$ the mapping $t\mapsto \Phi_{t,s}(x)$ is an (integral) solution to the ordinary differential equation
 \begin{equation}\label{e:regLagrange:ODE}
   \Phi_{t,s}(x) = x + \int_s^t u\bra*{\sigma,\Phi_{\sigma, s}(x)} \, d\sigma \qquad \text{for all } t\in [s,T] .
 \end{equation}
 The no-flux boundary condition~\eqref{2} ensures that the flow $\Phi_{t,s}$   maps $\Omega$ to $\Omega$. By the compressibility assumption~\eqref{20} (which can be retained under approximation) it follows that the Jacobian given by
\begin{equation*}
  J\Phi_{t,s} := \det \grad\Phi_{t,s} = \exp\bra*{ \int_s^t \div u\bra*{\sigma,\Phi_{\sigma,s}} \, d\sigma }
\end{equation*}
is bounded below by $\Lambda^{-1}$.
In particular, it holds that
 \begin{equation}\label{e:regLagrange:compress}
  | \Phi_{t,s}^{-1}(A)|\leq \Lambda |A| \qquad \text{for any Borel subset } A \text{ of } \Omega .
 \end{equation}
The constant $\Lambda$ is accordingly referred to as the compressibility constant of $\Phi$.
It is moreover well-known that the family $\set{\Phi_{t,s}}_{0\leq s \leq t \leq T}$ is an inhomogeneous semigroup, that is
\begin{equation*}
  \Phi_{t_1,s} \circ \Phi_{s,t_0} = \Phi_{t_1,t_0} \qquad\text{for all } 0\leq t_0 \leq s \leq t_1 \leq T.
\end{equation*}
The theory for Lagrangian flows has been extended to the low regularity framework by DiPerna and Lions \cite{DiPernaLions89}. See also the recent contribution of Ambrosio, Colombo and Figalli~\cite{AmbrosioColomboFigalli2015} for the theory maximal flows on bounded domains.

The superposition principle for a solution $\rho$ to the inhomogeneous continuity equation~\eqref{1} reads
\begin{equation}\label{e:sol:inhom}
 \rho_t = \bra*{ \phi_t}_\# \rho^0 + \int_0^t \bra*{\Phi_{t,s}}_\# f_s \, ds ,
\end{equation}
where we have set $\phi_t := \Phi_{t,0}$ and where the push forward of a map $\phi: \Omega \to \Omega$ is defined for any Borel function $\xi$ and Borel measure $\mu$ by
\[
 \int \xi \, d\phi_\# \mu:=  \int \xi \circ \phi \, d\mu  .
\]

We now proceed with the proof of Lemma~\ref{L3}.

\begin{proof}[Proof of Lemma \ref{L3}]
In view of the stability results obtained in \cite{DiPernaLions89} or \cite{Seis16a} and by a standard approximation argument, it is enough to consider vector fields that are smooth in the spatial variable.

Let $\Phi^{\dt}$ be the Lagrangian flow for the vector field $u_{\dt}$. Based on the superposition principle~\eqref{e:sol:inhom}, we have the following representations
\begin{align*}
  \rho(t) = \bra*{\phi_t}_\# \rho^0 + \int_0^t \bra*{ \Phi_{t,s}}_\# f(s) \, ds \quad\text{and}\quad \rho^{\dt}(t) = \bra*{\phi_t^{\dt}}_\# \rho^0 + \int_0^t \bra*{ \Phi_{t,s}^{\dt}}_\# f_{\dt}(s) \, ds.
\end{align*}
Let us introduce some convenient abbreviations for the homogeneous and inhomogeneous components of the solutions $\rho$ and $\rho^\dt$, respectively
\begin{align}
\rho_t^{\hom} &:= \bra[\big]{\phi_t}_\# \rho^0  &&\text{and} & \rho_t^{\dt,\hom} &:= \bra[\big]{\phi_t^{\dt}}_\# \rho^0 \ ; \label{e:rho_hom}\\
F_t  &:=  \int_0^t \bra[\big]{\Phi_{t,s}}_\# f_s \, ds &&\text{and}&  F_t^{\dt} &:=  \int_0^t \avint_{I^\dt_s} \bra[\big]{\Phi_{t,s}^{\dt}}_\# f_\sigma  \, d\sigma  \, ds \ , \label{e:F}
\end{align}
where for $s\in [t^n, t^{n+1})$, we set $I^\dt_s := [t^n,t^{n+1})$.
These definitions imply $\rho_t = \rho_t^{\hom} + F_t$ and likewise $\rho_t^\dt = \rho_t^{\dt,\hom} + F_t^{\dt}$.
Now, by using the triangle inequality~\eqref{5} for the Kantorovich--Rubinstein distance, we estimate
\[
D_r\bra[\big]{\rho(t^{\ell}),\rho^{\dt}(t^{\ell})} \le D_r\bra[\big]{\rho_{t^{\ell}}^{\hom},\rho_{t^{\ell}}^{\dt,\hom}} + D_r\bra[\big]{F_{t^{\ell}},F_{t^{\ell}}^{\dt}} .
\]
By the boundedness assumption~\eqref{ass:u:infty} on the vector field and the definition of the flows~\eqref{e:regLagrange:ODE}, it follows for any $s\in [0,T]$ and $t_0,t_1\in [s,T]$ the basic estimate
\begin{equation}\label{e:flow:infty}
  \max\set*{\norm{\Phi_{t_1,s} - \Phi_{t_0,s}}_{L^\infty} , \norm{\Phi_{t_1,s}^\dt - \Phi_{t_0,s}^\dt}_{L^\infty}} \leq \abs{t_1-t_0} \norm{u}_{L^\infty} .
\end{equation}
We first turn to the estimate of $D_r\bra[\big]{\rho_{t^{\ell}}^{\hom},\rho_{t^{\ell}}^{\dt,\hom}}$. Denoting by $\zo$ the corresponding optimal Kantorovich--Rubinstein potential, we have the estimate
\begin{align*}
D_r\bra[\big]{\rho_{t^{\ell}}^{\hom},\rho_{t^{\ell}}^{\dt,\hom}}
&\overset{\mathclap{\eqref{e:rho_hom}}}{=} \io \zo \bra*{ \bra[\big]{ \phi_{t^{\ell}} }_\# \rho^0 - \bra[\big]{\phi_{t^{\ell}}^\dt }_\# \rho^0} \, dx \\
&= \io \bra*{ \zo \circ \phi_{t^{\ell}} - \zo \circ \phi_{t^{\ell}}^\dt } \rho^0 \, dx \notag \\
&\le \io \log\bra[\bigg]{\frac{|\phi_{t^{\ell}}-\phi_{t^{\ell}}^\dt|}r+1} |\rho^0|\, dx,
\end{align*}
where in the last inequality we have used the Lipschitz property of $\zo$ in~\eqref{8}. It is thus enough to control the quantity on the right-hand side.   Using the triangle inequality in~$\R^d$, the concavity of the logarithm, and the definition of the flows, we first estimate
\[
\begin{aligned}
\MoveEqLeft \io \log\bra[\bigg]{\frac{|\phi_{t^{n+1}} -\phi_{t^{n+1}}^\dt|}r +1}|\rho^0|\, dx - \io \log\bra[\bigg]{\frac{|\phi_{t^n} -\phi_{t^n}^\dt|}r +1}|\rho^0|\, dx\\
&\le \io \frac{\abs[\big]{\int_{t^n}^{t^{n+1}} \bra[\big]{u(\sigma,\phi_\sigma) - u_\dt(\sigma,\phi_\sigma^\dt))\, d\sigma}}}{|\phi_{t^n}- \phi_{t^n}^\dt| +r} \, \abs{\rho^0} \, dx
\end{aligned}
\]
for any $n\in\llb 0,N\rrb$. The numerator in the integrand is controlled as follows: In view of the definition of $u_{\dt}$ and the Morrey-type estimate \eqref{7}, we have that
\begin{align*}
\MoveEqLeft \left|\int_{t^n}^{t^{n+1}} \bra[\big]{u(\sigma,\phi_\sigma) - u_\dt(\sigma,\phi_\sigma^\dt)}\, d\sigma\right|
 = \left| \frac1{\dt} \int_{t^n}^{t^{n+1}}\int_{t^n}^{t^{n+1}} \bra[\big]{u(\sigma,\phi_\sigma) - u(\sigma,\phi_\tau^\dt)}\, d\sigma\,d\tau \, \right|\\
&\lesssim \frac1{\dt}\int_{t^n}^{t^{n+1}}\int_{t^n}^{t^{n+1}} \bra*{(M\grad \bar u_\sigma)(\phi_\sigma) +(M\grad \bar u_\sigma)(\phi_\tau^\dt)} |\phi_\sigma-\phi_\tau^\dt|\, d\sigma\,d\tau.
\end{align*}
Using the estimate~\eqref{e:flow:infty}, we get $|\phi_\sigma-\phi_\tau^\dt|\le |\phi_{t^n}-\phi_{t^n}^\dt| + 2\dt \|u\|_{L^{\infty}}$. Combining the previous estimates thus yields
\begin{equation}\label{e:est:rho_hom:p1}
\begin{aligned}
&\int \log\bra[\bigg]{\frac{|\phi_{t^{n+1}} -\phi_{t^{n+1}}^\dt|}r +1}\abs{\rho^0}\, dx - \int \log\bra[\bigg]{\frac{|\phi_{t^n} -\phi_{t^n}^\dt|}r +1}\abs{\rho^0}\, dx  \\
&\lesssim\bra*{1+\frac{\dt\|u\|_{L^{\infty}}}{r}} \int_{t^n}^{t^{n+1}}\! \avint_{t^n}^{t^{n+1}} \!\! \io \bra[\big]{(M\grad\bar u_\sigma)\circ\phi_\sigma  +(M\grad\bar u_\sigma)\circ\phi_\tau^\dt}\abs{\rho^0}\, dx\,d\tau\,d\sigma.
\end{aligned}
\end{equation}
It remains to notice that by recalling~\eqref{e:rho_hom} and applying H\"older's inequality it follows
\begin{align*}
\int_{t^n}^{t^{n+1}} (M\grad\bar u_\sigma)(\phi_\sigma)\,\abs{\rho^0}\, dx\,d\sigma  &= \int_{t^n}^{t^{n+1}} (M\grad\bar u_\sigma) \, \bra{\phi_{\sigma}}_\#\abs{\rho^0} \, dx\,d\sigma \\
&\le \int_{t^n}^{t^{n+1}} \|M\grad\bar u_\sigma\|_{L^p}\, d\sigma \ \ \sup_{\sigma\in [0,T]}\norm[\big]{\bra{\phi_{\sigma}}_\#\abs{\rho^0}}_{L^q} .
\end{align*}
The terms on the right-hand side are estimated with the help of the fundamental inequality for maximal functions~\eqref{6} and the compressibility property~\eqref{e:regLagrange:compress}. The second term on the right-hand side of~\eqref{e:est:rho_hom:p1} is treated similarly by using in addition Lemma~\ref{lem:disc:norm}. Hence, summing over $n\in \llb 0,\ell-1\rrb$ yields
\[
D_r\bra[\big]{\rho_{t^{\ell}}^{\hom},\rho_{t^{\ell}}^{\dt,\hom}} \lesssim 1 + \frac{\dt \|u\|_{L^{\infty}}}r.
\]

Let us now focus on the term $D_r(F_{t^{\ell}},F_{t^{\ell}}^\dt)$. If $\zo$ is the Kantorovich potential corresponding to this transport distance, then we obtain analogously to the above computation
\begin{align*}
D_r(F_{t^{\ell}},F_{t^{\ell}}^\dt) &\;\overset{\mathclap{\eqref{e:F}}}{=}\; \int_0^{t^{\ell}} \avint_{I^\dt_s} \io \zo \bra[\Big]{ \bra[\big]{\Phi_{t^{\ell},s}}_\# f_s - \bra[\big]{\Phi_{t^{\ell},s}^\dt}_\# f_\sigma} \, dx \, d\sigma \, ds \notag \\
 &\;=\; \int_0^{t^{\ell}} \avint_{I^\dt_s} \io \bra*{\zo\circ \Phi_{t^{\ell},s} - \zo \circ \Phi_{t^{\ell},\sigma}^\dt } f_s\, dx \, d\sigma \, ds \notag \\
 &\;\le\; \int_0^{t^{\ell}}\avint_{I^\dt_s}  \io \log\bra[\Bigg]{\frac{|\Phi_{t^{\ell},s} - \Phi_{t^{\ell},\sigma}^\dt|}r + 1}\, |f_s|\, dx \,d\sigma \, ds\\\
 & =: \int_0^{t^{\ell}} \avint_{I^\dt_s}J(t,s,\sigma) \,d\sigma \, ds .
\end{align*}
The estimation of the expression on the right proceeds very analogously to the estimations above. We will thus only sketch it. In a first step, we obtain, similarly as in \eqref{e:est:rho_hom:p1} that
\begin{align*}
\MoveEqLeft
J(t^{n+1},s,\sigma) -J(t^n,s,\sigma)\\
&\lesssim\left(1+\frac{\dt\,\|u\|_{L^{\infty}}}r\right) \int_{t^{n}}^{t^{n+1}} \avint_{t^n}^{t^{n+1}}\io \left((M\grad u_{\tau})(\Phi_{\tau,s}) + (M\grad u_{\tau})(\Phi_{\gamma,\sigma}^{\delta})\right) |f_s|\, dx \, d\gamma \,d\tau,
\end{align*}
for any $n\in \llb 0,N\rrb$ with $s\le t^n$. By using the same arguments as above and summing over all $n\in \llb k,\ell-1\rrb$, where $k$ is such that $s,\sigma \in \left[t^{k-1},k\right)$, we find that
\begin{equation}
\label{51}
J(t^{\ell},s,\sigma) - J(t^k,s,\sigma) \lesssim \left(1+\frac{\dt\, \|u\|_{L^{\infty}}}r\right) \|u\|_{L^1(W^{1,p})} \|f_s\|_{L^q}.
\end{equation}
We observe that thanks to \eqref{e:flow:infty}, we have the estimate
\[
J(t^k,s,\sigma) \lesssim \frac{\dt\,\|u\|_{L^{\infty}}}r \|f_s\|_{L^1}.
\]
Plugging this bound into \eqref{51} and integrating over $\sigma$ and $s$ thus yields
\[
D_r(F_{t^{\ell}},F_{t^{\ell}}^{\delta})\lesssim 1+\frac{\dt\, \|u\|_{L^{\infty}}}r,
\]
which is what we aimed to show.
\end{proof}

\begin{proof}[Proof of Lemma \ref{L4}]
We split the term $\I^n$ into the sum $\I_1^n + \I^n_2$ with
\begin{align*}
\I_1^n &:= \int_{t^n}^{t^{n+1}} \int \grad \zo \cdot u\bra[\big]{\rho-\widehat \rho_{\dt,h}}\, dx\,dt ,\\
\I_2^n&:= \int_{t^n}^{t^{n+1}} \int \grad \zo \cdot u\bra[\big]{\widehat \rho_{\dt,h} - \rho_h^{n+1}}\, dx\,dt.
\end{align*}
We claim that
\begin{equation}\label{45}
\sum_{n=0}^{N-1} \I_1^n\lesssim 1\quad\mbox{and}\quad \sum_{n=0}^{N-1} \I^n_2 \lesssim\frac{\sqrt{\dt\, T}}r\|u\|_{L^{\infty}}.
\end{equation}

The proof  of the first  estimate in \eqref{45} is a direct consequence of estimate \eqref{16} that was established  in the proof of Lemma  \ref{L2}. (Here, it is important to notice that $\zo$ is a Kantorovich potential associated with $D_r(\rho,\widehat \rho_{\dt,h})$.) In fact, applying \eqref{16} together with the triangle inequality for $\|\cdot \|_{L^q}$ yields
\[
\I^n_1 \lesssim \Lambda^{\frac{1}{p}}\int_{t^n}^{t^{n+1}} \|u\|_{W^{1,p}}\, ds \bra*{\|\rho^0\|_{L^q} + \|f_{L^1(L^q)}}.
\]
Summing over $n$ and invoking the a priori estimates \eqref{18} and \eqref{47} gives the result.

For the second statement in \eqref{45}, we notice that 
because of
\[
  \widehat \rho_{\dt, h}(t) - \rho_h^{n+1} = \dt^{-1}(t^{n+1} - t)(\rho_h^n - \rho_h^{n+1}),
\]
we have
\[
 \I_2^n  = \avint_{t^n}^{t^{n+1}}(t^{n+1}-t) \int \grad\zo \cdot u\bra*{\rho_h^n-\rho_h^{n+1}}\, dx\,dt .
 \]
Using the Lipschitz bound for $\zo$~\eqref{49} and summing over $n$ gives
\[
\sum_n \I_2^n \le \frac{\dt\, \|u\|_{L^{\infty}}}r  \sum_n \int |\rho_h^n-\rho_h^{n+1}|\, dx.
\]
The conclusion is an immediate consequence of the temporal weak $BV$ estimate~\eqref{est:nabla:2} from Proposition~\ref{prop:nabla}.
\end{proof}

\begin{proof}[Proof of Lemma~\ref{L6}]
For the proof, we write $\zeta := \zo$ and $u^n := u(t^n)$. We first notice that  the anti-symmetry of $u_{KL}^n$ implies that
\[
\sum_K \zeta_K^n \sum_{L\sim K} |K\edge L|u_{KL}^n\frac{\rho_K^{n+1}+ \rho_L^{n+1}}2 = \sum_K \rho_K^{n+1} \sum_{L\sim K} |K\edge L| u_{KL}^n\frac{\zeta^n_K - \zeta^n_L}2.
\]
Moreover, the divergence theorem yields
\[
\zeta_K^n \sum_{L\sim K}|K\edge L|u_{KL}^n = \int_K\zeta_K^n\div u^n \, dx .
\]
Regarding the other term in $\II^n$, we have by integration by parts on each cell $K$ that
\[
\int_K \grad\zeta\cdot u^n \, dx = \sum_{L\sim K} \int_{K\edge L}\zeta u^n \cdot \nu_{KL}\, d\Ha^{d-1} - \int_K\zeta\div u^n\, dx,
\]
It thus follows that  $\II^n$ can be rewritten as
\begin{align*}
\II^n & = \dt \sum_K \rho_K^{n+1} \sum_{L\sim K} \bra*{\int_{K\edge L}\zeta^n \,u^n \cdot \nu_{KL}\, d\Ha^{d-1}- |K\edge L|u_{KL}^n\frac{\zeta_K^n+\zeta_L^n}2}\\
&\quad - \dt\sum_K \rho^{n+1}_K \int_K\bra*{\zeta^n-\zeta_K^n}\div u^n \, dx,
\end{align*}
where $\zeta^n = \zeta^n(x)$ denotes the time average of $\zeta$ over $\bra*{t^n,t^{n+1}}$. We can furthermore decompose this expression
\begin{align*}
\II^n &= \dt \sum_K \rho_K^{n+1} \! \sum_{L\sim K}\int_{K\edge L} \! \zeta^n \bra[\bigg]{\bra[\big]{u^n-u_K^n}- \avint_{K\edge L}\! \bra[\big]{u^n-u_K^n}\, d\Ha^{d-1}}\cdot \nu_{KL}\, d\Ha^{d-1} \\
&\quad + \dt \sum_K\rho_K^{n+1} \sum_{L\sim K} |K\edge L|u_{KL}^n \bra*{\avint_{K\edge L} \zeta^n \, d\Ha^{d-1} - \frac{\zeta_K^n + \zeta_L^n}2}\\
&\quad- \dt \sum_K \rho_K^{n+1} \int_K \bra*{\zeta^n - \zeta_K^n}\div u^n \, dx,
\end{align*}
and we set $\II^n =: \II^n_1+  \II^n_2 + \II^n_3$, accordingly. To estimate $\II_1^n$, we notice that we can smuggle in the constant function $\zeta_K^n$ in each boundary integral, leading to
\[
\II^n_1  \le 2 \dt  \sum_K\rho_K^{n+1}\|\zeta^n-\zeta^n_K\|_{L^{\infty}(K)} \int_{\partial K} |u^n-u_K^n|\, d\Ha^{d-1}.
\]
Thanks to the Lipschitz property of $\zo$, cf.~\eqref{49}, it holds that $|\zeta^n - \zeta_K^n|\le h/r$ uniformly in $K$. We now combine the trace and Poincar\'e estimate~\eqref{3}, to the effect that
\[
\|u^n-u_K^n\|_{L^1(\partial K)} \lesssim h^{-1}\|u^n-u_K^n\|_{L^1(K)} + \|\grad u^n\|_{L^1(K)} \lesssim \|\grad u^n\|_{L^1(K)}.
\]
We are thus left with
\[
\II^n_1 \lesssim \frac{h \, \dt}r \sum_K \rho_K^{n+1} \|\grad u^n\|_{L^1(K)} = \frac{h\, \dt}r \io \rho_h^{n+1}|\grad u^n|\, dx.
\]
Summing over $n$ and using H\"older's inequality and the a priori bound \eqref{47} yields that
\[
\sum_n \II^n_1 \lesssim \frac{h}r.
\]
To estimate $\II^n_2$, we first  notice that for symmetry reasons, we have
\[
\II^n_2 = \dt\sum_K\sum_{L\sim K} \frac{\rho_K^{n+1} - \rho_L^{n+1}}2 |K\edge L| u_{KL}^n \bra[\bigg]{\avint_{K\edge L}\zeta^n\, d\Ha^{d-1} -\frac{\zeta_K^n +\zeta_L^n}2}.
\]
Using again the Lipschitz property for $\zo$, cf.~\eqref{49}, the latter is bounded as follows:
\[
\II^n_2\le \frac{\dt\, h}r \sum_K\sum_{L\sim K} |K\edge L||u_{KL}^n| |\rho_K^{n+1}-\rho_L^{n+1}|.
\]
We apply the spatial weak $BV$ estimate~\eqref{est:nabla:1} from Proposition~\ref{prop:nabla} showing
\[
  \sum_n\II^n_2 \lesssim \frac{\sqrt{h T \norm{u}_{L^\infty}}}r .
\]
It remains to investigate $\II^n_3$. Once again, we use the Lipschitz bound of $\zo$ in~\eqref{49}, Hölder's inequality and bound $\div u$ by $\nabla u$ up to a dimension dependent constant to
estimate
\[
\II_3^n \le \frac{\dt \, h}r \|\rho_h^{n+1}\|_{L^q}\|\nabla u^n\|_{L^{p}}.
\]
Summing over $n$ and using the a priori estimate \eqref{e:Lq:stability} yields
\[
\sum_n \II_3^n \lesssim \frac{h}{r}.
\]
This  concludes the proof of Lemma \ref{L6}.
\end{proof}

\begin{proof}[Proof of Lemma~\ref{L7}]
  From the Lipschitz bound~\eqref{49} we deduce for any two neighboring cells $K$ and $L$ that
  \begin{equation*}
    |(\zo)_K^{n+1} - (\zo)_L^{n+1}|\le \avint_K\avint_L |\zo(x) - \zo(y)|\, dx\,dy \lesssim \frac{h}{r} .
  \end{equation*}
  Therewith, we estimate after doing a symmetrization of the sum:
  \begin{align*}
    \III^n &= \frac{\dt}{4} \sum_{K} \sum_{L\sim K} \abs{K \edge L} \abs{u_{KL}^n} \bra*{ \rho_K^{n+1} - \rho_L^{n+1}} \bra*{ (\zo)_K^{n+1} - (\zo)_L^{n+1}} \\
    &\lesssim \frac{\dt \, h}{r} \sum_{K} \sum_{L\sim K} \abs{K \edge L} \abs{u_{KL}^n} \abs*{ \rho_K^{n+1} - \rho_L^{n+1}} .
  \end{align*}
  The proof concludes by summing over $n\in\llb 0, N-1\rrb$ and applying the spatial weak $BV$ estimate~\eqref{est:nabla:1} from Proposition~\ref{prop:nabla}.
\end{proof}

\appendix
\section{The \texorpdfstring{$q$}{q}-mean}\label{s:appendix}
\noindent We briefly describe some helpful estimates for the $q$-mean defined for $q>1$ by
\[
\theta_q: \R_+ \times \R_+ \to \R_+ \qquad\text{with}\qquad \theta_q(a,b) := \frac{q-1}{q} \frac{a^q - b^q}{a^{q-1} -b^{q-1}} .
\]
\begin{enumerate}[ (i) ]
 \item The function $\theta_q$ has the following integral representation
\begin{equation*}
  \theta_q(a,b) = \int_0^1 \bra*{ (1-s) a^{q-1} + s b^{q-1}}^{\frac{1}{q-1}} \, ds .
\end{equation*}
 \item The function $\theta_q$ is $1$-homogeneous: for any $c>0$ it holds $\theta_q(c\,a, c\,b) = c\,\theta_q(a,b)$.
 \item For any positive numbers $a\ne b$ is $q\mapsto \theta_q(a,b)$ strictly increasing.
 \item The function $(a,b)\mapsto \theta_q(a,b)$ is concave for $q\in (1,2)$ and convex for $q\in (2,\infty)$.
 \item For any $a,b>0$ it holds
\[
     \abs[\big]{ \theta_2\bra*{a,b}  - \theta_q\bra*{a  , b}} \leq \frac{\abs{q-2}}{q} \frac{\abs{a-b}}{2}.
\]
\end{enumerate}
\begin{proof} For the identity (i) let $t_q(a,b;s) := \bra*{ (1-s) a^{q-1} + s b^{q-1}}^{\frac{1}{q-1}}$ denote the integrand on the right hand side. Then, by a straightforward calculation it follows $\partial_s t_q(a,b;s) = \frac{b^{q-1}-a^{q-1}}{q-1} t(s)^{2-q}$ and hence $T_q(a,b;s) := \frac{q-1}{q} \frac{t(s)^q}{b^{q-1}-a^{q-1}}$ is its primitive from which (i) follows.

The $1$-homo\-geneity as stated in (ii) follows immediately from the definition.

For proving (iii), we note that $t_q(a,b;s)$ is the $\ell^{q-1}$ norm on the probability space $\bra[\big]{\set{0,1},(1-s) \delta_0 + s \delta_1}$ for a function $f:\set{0,1}\to \R$ taking values $f(0)=a$ and $f(1)=b$. Hence, the statement is a consequence of the ordering of the $\ell^p$ spaces, which extends to any value $p\in \R$.

The property (iv) follows by calculating the Hessian of $(a,b)\mapsto t_s(a,b):$
 \begin{align*}
   \Hess_{a,b} t_q(a,b;s) &= (q-2) \; (1-s)s \; \bra[\big]{ a\, b \, t_s(a,b;s)}^{q-3}
   \begin{pmatrix} a^2 &  ab \\ ab & b^2 \end{pmatrix} .
 \end{align*}
 Now, one immediately recovers that $(a,b)\mapsto t_q(a,b;s)$ is negative semidefinite for $q\in [0,2)$ and positive semidefinite for $q>2$.

 For (v), we can assume by symmetry and $1$-homogeneity that $a\in (0,1)$ and $b=1$. From (iv), we have that the mapping $(0,1) \ni a\mapsto \theta_q(a,1)$ is concave for $q\in (1,2)$ and convex for $q>2$. Let us first assume $q\in (1,2)$, then by convexity of $a\mapsto \theta_2(a,1)-\theta_q(a,1)$, we estimate using the secant inequality between the points~$0$ and~$1$:
 \begin{align*}
   0 \;\overset{\mathclap{\text{(iii)}}}{\leq}\; \theta_2(a,1) - \theta_q(a,1) \leq  \bra*{\theta_2(0,1) - \theta_q(0,1)} \bra{1-a} = \bra*{\frac{1}{2} - \frac{q-1}{q}} \bra{1-a} .
 \end{align*}
 For $q>2$, we apply the same argument to the convex function $a\mapsto \theta_q(a,1)-\theta_2(a,1)$.
\end{proof}

\bibliographystyle{abbrv}
\bibliography{coarsening}
\end{document}